%% file: Equivalence-Arxiv-2019-Jan.tex
\documentclass[10pt]{article}
\usepackage{microtype}
\usepackage[pdfpagelabels,breaklinks,colorlinks,linkcolor=red,citecolor=green,urlcolor=red]{hyperref}
\usepackage[textwidth=6.5in,textheight=8.6in]{geometry}
\usepackage{placeins}
\usepackage{epsfig}
\usepackage{latexsym}
\usepackage{amssymb}
\usepackage{amsmath}
\usepackage{algorithm}
\usepackage{amsthm}
\usepackage{caption}
\usepackage{amsfonts}
\usepackage{paralist,dsfont}
\usepackage{enumitem}
\usepackage{enumerate}
\usepackage{cases}
\usepackage[normalem]{ulem}
\usepackage{color,xcolor}
\usepackage{pdfpages}
\usepackage{listings}
\usepackage{longtable}
\usepackage{lscape}
\usepackage{multirow}

\def\mc{\multicolumn}
\def\norm#1{\|#1\|}

\DeclareMathAlphabet{\mathpzc}{OT1}{pzc}{m}{it}
\DeclareMathAlphabet{\mathbfcal}{OMS}{cmsy}{b}{n}

\colorlet{blue}{blue!90!black}
\colorlet{red}{red!50!black}
\colorlet{green}{green!50!black}
\newcommand{\blue}[1]{\begin{color}{blue}#1\end{color}}

\newcommand{\red}[1]{\begin{color}{red}#1\end{color}}
\newcommand{\green}[1]{\begin{color}{green}#1\end{color}}

\newtheorem{definition}{Definition}
\newtheorem{assumption}{Assumption}
\newtheorem{lemma}{Lemma}
\newtheorem{theorem}{Theorem}
\newtheorem{proposition}{Proposition}

\newtheorem{remark}{Remark}
\numberwithin{equation}{section}
\numberwithin{remark}{section}

\numberwithin{equation}{section}
\numberwithin{theorem}{section}
\numberwithin{definition}{section}
\numberwithin{proposition}{section}
\numberwithin{remark}{section}
\numberwithin{assumption}{section}
\numberwithin{lemma}{section}

\DeclareMathOperator*{\argmin}{arg\,min}
\DeclareMathOperator*{\ri}{ri}

\DeclareMathOperator*{\diag}{Diag}
\DeclareMathOperator*{\dom}{dom}
\DeclareMathOperator*{\range}{Range}
\def\norm#1{\|#1 \|}
\def\inprod#1#2{\langle#1,\,#2 \rangle}
\def\what{\widehat}
\def\prox{{\rm Prox}}
\def\dist{{\rm dist}}
\def\sgs{{\sf sGS}}
\def\mc{\multicolumn}
\def\dH{{\mathds H}}
\def\dU{{\mathds U}}
\def\dW{{\mathds W}}
\def\dX{{\mathds X}}
\def\dY{{\mathds Y}}
\def\dZ{{\mathds Z}}
\def\dV{{\mathds V}}
\def\Sn{{\mathds S}^n}
\def\bA{{\mathbfcal{A}}}
\def\bK{{\mathbf K}}
\def\cA{{\mathcal A}}
\def\cB{{\mathcal B}}

\def\cD{{\mathcal D}}

\def\cF{{\mathcal F}}
\def\cG{{\mathcal G}}
\def\cI{{\mathcal I}}

\def\cL{{\mathcal L}}
\def\cM{{\mathcal M}}
\def\cN{{\mathcal N}}
\def\cO{{\mathcal O}}
\def\cP{{\mathcal P}}
\def\cQ{{\mathcal Q}}
\def\cR{{\mathcal R}}
\def\cS{{\mathcal S}}
\def\cT{{\mathcal T}}
\def\sig{\sigma}
\def\S{\mathds{S}}
\def\R{\mathds{R}}
\def\ds{\displaystyle}
\def\[{\begin{equation}}
\def\]{\end{equation}}

\title{\bf \Large On the Equivalence of Inexact Proximal ALM and ADMM for a Class of
Convex Composite Programming}
\date{\today}
\author{
Liang Chen\thanks{College of Mathematics and Econometrics, Hunan University, Changsha, 410082, China (\url{chl@hnu.edu.cn}), and
Department of Applied Mathematics, The Hong Kong Polytechnic University, Hung Hom, Kowloon, Hong Kong (\url{liangchen@polyu.edu.hk}).
The research of this author was supported by the National Natural Science Foundation of China (11801158, 11871205) and the Fundamental Research Funds for the Central Universities in China.
}
 \quad
Xudong Li\thanks{School of Data Science, Fudan University, Shanghai 200433, China, and  Shanghai Center for Mathematical Sciences, Fudan University, Shanghai 200433, China (\url{lixudong@fudan.edu.cn}).
The research of this author was supported by the Fundamental Research Funds for the Central Universities in China.
}
\quad
Defeng Sun\thanks{Department of Applied Mathematics, The Hong Kong Polytechnic University, Hung Hom, Kowloon, Hong Kong
(\url{defeng.sun@polyu.edu.hk}). The research of this author was supported in part by a start-up research grant from the Hong Kong Polytechnic University.}
\ \ and \ \
Kim-Chuan Toh\thanks{Department of Mathematics, and Institute of Operations Research and Analytics, National University of Singapore, 10 Lower Kent Ridge Road, Singapore 119076 (\url{mattohkc@nus.edu.sg}).
The research of this author was supported in part by the Ministry of Education, Singapore, Academic Research Fund (R-146-000-257-112).
}
}
\date{March 28, 2018; Revised on Jan 28, 2019}

\begin{document}
\maketitle

%%%%%%%%%%%%%%%%%%%%%%%%%%%%% ABSTRACT %%%%%%%%%%%%%%%%%%%%%%%%%%%%%
\begin{abstract}
In this paper, we show that for a class of linearly constrained convex composite optimization problems, an (inexact) symmetric Gauss-Seidel based majorized multi-block proximal alternating direction method of multipliers (ADMM) is equivalent to an {\em inexact} proximal augmented Lagrangian method (ALM). This equivalence not only provides new perspectives for understanding some ADMM-type algorithms but also supplies meaningful guidelines on implementing them to achieve better computational efficiency. Even for the two-block case, a by-product of this equivalence is the convergence of the whole sequence generated by the classic ADMM with a step-length that exceeds the conventional upper bound of $(1+\sqrt{5})/2$, if one part of the objective is linear. This is exactly the problem setting in which the very first convergence analysis of ADMM was conducted by Gabay and Mercier in 1976, but, even under notably stronger assumptions, only the convergence of the primal sequence was known. A collection of illustrative examples are provided to demonstrate the breadth of applications for which our results can be used. Numerical experiments on solving a large number of linear and convex quadratic semidefinite programming problems are conducted to illustrate how the theoretical results established here can lead to improvements on the corresponding practical implementations.
\end{abstract}
\medskip
{\small
\begin{center}
\parbox{0.95\hsize}{{\bf Keywords.}\;
Alternating direction method of multipliers;
Augmented Lagrangian method;
Symmetric Gauss-Seidel decomposition;
Proximal term}
\end{center}
\begin{center}
\parbox{0.95\hsize}{{\bf AMS Subject Classification.}\; 90C25; 65K05; 90C06; 49M27; 90C20}
\end{center}}

%\color{black}

%%%%%%%%%%%%%%%%%%%%%%%%%%%%%%%%%%%%%%%%%%%%%

\section{Introduction}
Let $\dX$, $\dY$ and $\dZ$ be three finite-dimensional real Hilbert spaces each endowed with an inner product denoted by $\langle\cdot,\cdot\rangle$ and its induced norm denoted by $\|\cdot\|$, where $\dY:=\dY_1\times\cdots\times\dY_s$ is the Cartesian product of $s$ finite-dimensional real Hilbert spaces $\dY_i,i=1,\ldots,s,$ each endowed with the inner product, as well as the induced norm, inherited from $\dY$. For any given $y\in\dY$, we can write $y=(y_1; \ldots; y_s)$ with $y_i\in\dY_i,\forall\, i=1,\ldots, s$.
Here, and throughout this paper, we use the notation $(y_1;\ldots; y_s)$ to mean that the vectors $y_1,\ldots,y_s$
are written symbolically in a column format.

In this paper, we shall focus on the following multi-block convex composite optimization problem
\[
\label{probmulti}
\min_{y\in\dY,z\in\dZ}
\left\{
p(y_1)+f(y)
-\langle b,z\rangle
\mid
\cF^*y+\cG^*z=c\right\},
\]
where $p:\dY_1\to(-\infty,+\infty]$ is a (possibly nonsmooth) closed proper convex function,
$f:\dY\to(-\infty,+\infty)$ is a continuously differentiable convex function whose gradient is Lipschitz continuous,
$b\in\dZ$ and $c\in\dX$ are the given data,
and $\cF^*$ and $\cG ^*$ are the adjoints of the given linear mappings $\cF:\dX\to\dY$ and $\cG :\dX\to\dZ$, respectively.
Despite the simple appearance of problem \eqref{probmulti}, we shall see in the next section that this model actually encompasses various important classes of convex optimization problems in both classical core convex programming as well as recently emerged models from a broad range of real-world applications.
A quintessential example of problem \eqref{probmulti} is the dual of the following convex composite quadratic programming
\[
\label{modelexample}
\min_{x}
\left\{\psi(x)+\frac{1}{2}\langle x,\cQ x\rangle-\langle c,x\rangle
\mid
\cG x=b
\right\},
\]
where $\psi:\dX\to(-\infty,+\infty]$ is a closed proper convex function, $\cQ:\dX\to\dX$ is a self-adjoint positive semidefinite linear operator, $\cG:\dX\to\dZ$ is a linear mapping,
and $c\in\dX$, $b\in\range(\cG)$ (i.e., $b$ is in the range space of the linear operator $\cG$) are the given data.
The dual of problem \eqref{modelexample} in the minimization form can be written as follows:
\[
\label{modelexampled}
\min_{y_1,y_2,z}
\left\{
\psi^*(y_1)
+\frac{1}{2}\langle y_2,\cQ y_2\rangle
-\langle b,z\rangle
\ \big|\ y_1+\cQ y_2- \cG^*z=c
\right\},
\]
where $\psi^*$ is the Fenchel conjugate of $\psi$,
$y_1\in\dX$, $y_2\in\dX$ and $z\in\dZ$, so that problem \eqref{modelexampled} constitutes an instance of problem \eqref{probmulti}.

To solve problem \eqref{probmulti}, one of the most preferred approaches is the augmented Lagrangian method (ALM) initiated by Hestenes \cite{hestenes} and Powell \cite{powell}, and elegantly studied for general (without taking
into account of the multi-block structure) convex optimization problems in the seminal work of Rockafellar \cite{roc76b}.
Given a penalty parameter $\sigma>0$, the augmented Lagrangian function corresponding to problem \eqref{probmulti} is defined by
$$
\begin{array}{r}
\ds
L_{\sigma}(y,z;x):=p(y_1)+f(y)-\langle b,z\rangle
+\langle x,\cF^*y+\cG^*z-c\rangle+\frac{\sig}{2}\|\cF^*y+\cG^*z-c\|^2,\quad
\\[3mm]
\quad \forall\,(x,y,z)\in\dX\times\dY\times\dZ.
\end{array}
$$
Starting from a given initial multiplier $x^0\in\dX$, the ALM performs the following steps at the $k$-th iteration:
\begin{description}
\item[(1)] compute $(y^{k+1},z^{k+1})$ to (approximately) minimize the function $L_{\sigma}(y, z; x^k)$, and

\smallskip
\item[(2)] update the multipliers $x^{k+1}:=x^k+\tau\sigma(\cF^*y^{k+1}+\cG^*z^{k+1}-c)$, where $\tau\in(0,2)$ is the step-length.
\end{description}
While one would really want to solve $\min_{y,z} L_{\sigma}(y,z;x^k)$ as it is without modifying
the augmented Lagrangian function, it can be expensive to
minimize $L_{\sigma}(y,z;x^k)$ with respect to both $y$ and $z$ simultaneously,
due to the coupled quadratic term in $y$ and $z$. Thus, in practice, unless the ALM is converging rapidly, one would generally
want to replace the augmented Lagrangian subproblem with an easier-to-solve surrogate by
modifying the augmented Lagrangian function to decouple the minimization with
respect to $y$ and $z$. Such a modification is especially desirable during the initial phase of the ALM
when its local superlinear convergence has yet to kick in.
The most obvious approach to decouple the subproblem for obtaining $(y^{k+1}, z^{k+1})$ is to add
to $L_\sig(y,z;x^k)$ the proximal term $\frac{\sig}{2}\norm{ (y;z) -(y^k; z^k)}^2_\Lambda$,
where $\Lambda = \lambda^2\cI - (\cF; \cG)(\cF; \cG)^*$ with $\lambda$ being the largest singular value of $(\cF;\cG)$ and $\cI$ being the identity operator in $\dY\times\dZ$.
However, such a modification to the augmented Lagrangian function is generally too drastic and has the undesirable effect of significantly slowing down the convergence of the ALM \cite[Section 7]{chenl2015}.
This naturally leads us to the important question on what is an appropriate proximal term
to add to $L_\sig(y,z;x^k)$ such that the ALM subproblem is easier to solve while at the same time it is less drastic than the obvious choice we have just mentioned in the previous sentences.

We shall show in this paper that by adding an appropriately designed proximal term
to $L_\sig(y,z;x^k)$, we can reduce the computation of the modified ALM subproblem
to sequentially updating $y$ and $z$ via computing
$$
y^{k+1} \approx \min_y\left\{ L_\sig(y, z^k; x^k)\right\}
\quad\mbox{and}\quad
z^{k+1} \approx \min_z \left\{ L_\sig(y^{k+1},z; x^k)\right\}
.$$
The reader would have observed that the resulting proximal ALM updating scheme is the same
as the classic two-block ADMM (pioneered by Glowinski and Marroco \cite {glo75} and Gabay and Mercier \cite{gabay1976}) that is applied to problem \eqref{probmulti}. However, there is a \emph{crucial difference} in that our convergence result holds true for the step-length $\tau$ in the range $(0,2)$, whereas the classic two-block ADMM only allows the step-length to be in the interval $\big(0,(1+\sqrt{5})/2\big)$ if the convergence of the full sequence generated by the algorithm is required.
It is important to note that even with the sequential minimization of $y$ and $z$ in the modified ALM subproblem,
the minimization subproblem with respect to $y$ can still be very difficult to solve due to
the coupling of the blocks $y_1,\ldots, y_s$ in \eqref{probmulti}.
One of the main contributions we made in this paper is to show that by majorizing the function $f(y)$ at $y^k$ with a quadratic function and by
adding an extra proximal term that is derived based on the block symmetric Gauss-Seidel (sGS) decomposition
theorem \cite{lisgs} for the quadratic term associated with $y$, we are able to update the sub-blocks in $y$
individually
in a symmetric Gauss-Seidel fashion. A crucial implication of this result is that the (inexact) block sGS decomposition based multi-block majorized ADMM is equivalent to an inexact majorized proximal ALM.
Consequently, we are able to prove the convergence of the whole sequence generated by the former even when the step-length is in the range $(0,2)$.

In this paper, we shall not delve into the vast literature on both ALM and ADMM, as well as their variants, and their relationships to the proximal point method and operator splitting methods.
They are simply too abundant for us to list even a few of them here. Thus, we shall only refer to those that are most relevant for our work in this paper.
Here we should mention that many attempts have been made in recent years on
designing convergent multi-block ADMM-type algorithms that can outperform the directly extended multi-block (proximal ADMM) numerically.
While the latter is not guaranteed to converge even under the strong assumption that $f\equiv0$,
paradoxically its practical numerical performance is often better than many convergent variants that
have been developed in the past; see for example \cite{sty2015}.
Against this backdrop, we should mention that the ADMM-type algorithms that have been progressively designed in \cite{sty2015,lixd2014,chenl2015}
not only come with convergence guarantee but they have also been demonstrated to have superior numerical performance than the directly extended ADMM, at least for a large number of convex conic programming problems.
More recently, those algorithms have found applications in various areas \cite{bai,BAIQI2016,DQi2016,ferr,lam,liqsdpnal,wangzou,yangaoteo,yst2015}.
Among those algorithms, the most general and versatile one is the recently developed inexact majorized multi-block proximal ADMM in Chen et al. \cite{chenl2015}, which we shall briefly describe in the next paragraph.

Under the assumption that the gradient of $f$ is Lipschitz continuous, we know that
one can specify a fixed self-adjoint positive semidefinite linear operator $\what\Sigma^f:\dY\to\dY$ and define at each $y'\in\dY$ the following convex quadratic function
\[
\label{mfy}
\what f(y,y'):=f(y')+\langle\nabla f(y'),y-y'\rangle
+\frac{1}{2}\|y-y'\|^2_{\what\Sigma^f}\, ,
\quad\forall\, y\in\dY,
\]
such that
$$f(y)\le \what{f}(y,y'),\quad \forall y,y'\in\dY\quad\mbox{and}\quad
f(y')=\what{f}(y',y'),\quad \forall y'\in\dY.$$
Thus, we say that at each $y'\in\dY$, the function $\what f(\cdot,y')$ constitutes a majorization of the function $f$.
Let $\sigma>0$ be the penalty parameter.
Based on the notion of majorization described above, the \emph{majorized} augmented Lagrangian function of problem \eqref{probmulti} is defined by
\[
\label{lagr}
\begin{array}{l}
\cL_\sigma\big(y,z;(x,y')\big):=p(y_1)+\what f(y,y')-\langle b,z\rangle+\langle \cF^*y+\cG^*z-c,x\rangle
\\[3mm]
\qquad\qquad\qquad\quad\ds
+\frac{\sigma}{2}\|\cF^*y+\cG^*z-c\|^2,
\quad
\forall\,(y,z,x,y')\in\dY\times\dZ\times\dX\times\dY.
\end{array}
\]
Let $(x^0,y^0,z^0)\in\dX\times\dY\times\dZ$ be a given initial point with $y_1^0 \in \dom p$,
and $\cD_i:\dY_i\to\dY_i$, $i=1,\ldots,s$ be the given self-adjoint linear operators, for the purpose of facilitating the computations of the subproblems.
For convenience, we denote for any $y=(y_1;\ldots;y_s)\in\dY_1\times\cdots\times\dY_s$,
$$y_{< i}:=(y_1;\ldots;y_{i-1})\quad\mbox{and}\quad
y_{> i}:=(y_{i+1};\ldots;y_s),
\quad\forall\, i=1,\ldots,s.$$
Then, the $k$-th step of the (inexact) block sGS decomposition based majorized multi-block
proximal ADMM in \cite{chenl2015}, when applied to problem \eqref{probmulti}, takes the following form
\[
\label{iadmm}
\left\{
\begin{array}{l}
y_{i}^{k+\frac{1}{2}}
\approx{\argmin\limits_{y_{i}\in\dY_{i}}}
\left\{\cL_\sigma
\Big(\big(y^{k}_{< i};y_{i}; y^{k+\frac{1}{2}}_{> i}\big),z^{k};(x^k,y^k)\Big)
+\frac{1}{2}\|y_{i}-y_{i}^k\|_{\cD_{i}}^2 \right\},
\mbox{ $i=s,\ldots,2$\,;}
\\[3mm]
y_{i}^{k+1}
\approx{\argmin\limits_{y_{i}\in\dY_{i}}}
\left\{\cL_\sigma
\Big(\big(y^{k+1}_{< i};y_{i}; y^{k+\frac{1}{2}}_{> i}\big),z^{k};(x^k,y^k)\Big)
+\frac{1}{2}\|y_{i}-y_{i}^k\|_{\cD_{i}}^2 \right\}, \mbox{ $i=1,\ldots,s$}\,;
\\[3mm]
z^{k+1}\approx\argmin\limits_{z\in\dZ}\cL_{\sigma}\Big(y^{k+1},z;(x^k,y^k)\Big);
\\[3mm]
x^{k+1}=x^k+\tau\sigma\big(\cF^*y^{k+1}+\cG^* z^{k+1}-c\big),
\end{array}
\right.
\]
where $\tau\in\left(0,(1+\sqrt{5})/2\right)$ was allowed in \cite{chenl2015}.
As one can observe from \eqref{lagr} and \eqref{iadmm}, the quadratic majorization technique
in Li {et al.} \cite{limin} was used to replace the original augmented Lagrangian function
by the majorized augmented Lagrangian function.
This in turn enables us to employ the inexact block sGS decomposition technique in Li {et al.} \cite{lisgs} to sequentially update the sub-blocks of $y$ individually.
More importantly, the algorithm is highly flexible in that
all the subproblems are allowed to be solved approximately to overcome possible numerical obstacles
such as, for example,
when iterative solvers must be employed to solve {large-scale} linear systems to overcome extreme
memory requirement and prohibitive computing cost.
It has already been demonstrated in \cite{chenl2015} that the inexact block sGS decomposition based multi-block ADMM is far superior to the directly extended ADMM in solving high-dimensional linear and
convex quadratic semidefinite programming with the
step-length in \eqref{iadmm} being restricted to be less than $(1+\sqrt{5})/2.$

Our focus in this paper is to investigate whether the framework in \eqref{iadmm} can
be proven to be convergent for problem \eqref{probmulti} when the step-length $\tau$ is in the range $(0,2)$.
In particular, we will show that the inexact block sGS decomposition based multi-block ADMM \eqref{iadmm} is equivalent to an {\em inexact} majorized proximal ALM in the sense that computations of $y^{k+1}$, $z^{k+1}$ and $x^{k+1}$ in \eqref{iadmm} can
equivalently be written as
$$
\left\{
\begin{array}{l}
\left(y^{k+1},z^{k+1}\right)
\approx
\argmin\limits_{(y,\, z)\in\dY\times\dZ}
\Big\{\cL_\sigma
\big(y,z;(x^k,y^k)\big)
+\frac{1}{2}\|\big(y;z\big)-\big(y^k;z^k\big)\|_{\cT}^2 \Big\};
\\[5mm]
x^{k+1}=x^k+\tau\sigma\big(\cF^*y^{k+1}+\cG^* z^{k+1}-c\big),
\end{array}
\right.
$$
where $\cT:\dY\times\dZ\to\dY\times\dZ$ is a self-adjoint (not necessarily positive definite) linear operator whose precise definition will be given later, and $\|(y;z)\|^2_\cT:=\langle (y;z),\cT (y;z)\rangle,\,\forall\,(y,z)\in\dY\times\dZ$.
This connection not only provides new theoretical perspectives for analyzing multi-block ADMM-type algorithms, but also has the potential of allowing them to achieve even better computational efficiency since
a larger step-length beyond $(1+\sqrt{5})/2$ can now be taken in \eqref{iadmm}, without adding any extra conditions or any additional verification steps such as those extensively used in \cite{sty2015,lixd2014,chenl2015,chennote}.

The main contributions of this paper are as follows.
\begin{itemize}
\item
We derive the equivalence of an (inexact) block sGS decomposition based multi-block majorized proximal ADMM to an inexact majorized proximal ALM, and establish the global and local convergence properties of the latter with the step-length
$\tau \in (0,2)$. As a result, the global and local convergence properties of the former
even with $\tau\in(0,2)$ are also established.
\item
Even for the most conventional two-block case, we are able for the first time to rigorously characterize the connection between ADMM and proximal ALM.
Note that given the form of the updating rules of the classic ADMM and ALM, although it is natural to view ADMM as an approximate version of the ALM, this is not completely true as can be seen from our analysis in this paper.
Indeed, to alleviate the difficulty of solving the subproblems in the ALM, the classic ADMM uses a single cycle of
the Gauss-Seidel block minimization to replace the full minimization of the augmented Lagrangian function in the ALM. This viewpoint in fact
motivated the study of the classic ADMM in the very first paper \cite{glo75}. However, as was mentioned in \cite{eckstein2012,eckstein2015}, there were no known results in
 quantifying this interpretation.

\item
As a by-product of the second contribution, this paper gives an affirmative answer to the open question on whether the dual sequence generated by the classic ADMM with $\tau\in(0,2)$ is convergent if one of the two functions in the objective is linear\footnote{This question was first resolved in \cite{sty2015} when the initial multiplier $x^0$ satisfies $\cG x^0-b=0$ and all the subproblems are solved exactly.}.
This is the problem setting of the very first proof for the ADMM in Gabay and Mercier \cite[Theorem 3.1]{gabay1976} in which the dual sequence is only guaranteed to be bounded, even under very strong assumptions.
The later proof of Glowinski \cite[Chapter 5, Theorem 5.1]{glo80} established stronger results than \cite{gabay1976} but it requires
$\tau\in\left(0,(1+\sqrt{5})/2\right)$. Thereafter, only the latter interval, and especially the unit step-length, has been considered.
In fact, in a rigorous proof presented recently in \cite{chennote} for the classic two-block ADMM with $\tau\in\left(0,\big(1+\sqrt{5}\big)/2\right)$, it was shown that the convergence of the dual sequence can be guaranteed under pretty weak conditions but the convergence of the primal sequence requires more.
Hence, it is of much theoretical interest to clarify whether the dual sequence is convergent if the objective contains a linear part while $\tau\ge\big(1+\sqrt{5}\big)/{2}$.

\item
We provide a fairly general criterion for choosing the possibly indefinite\footnote{One may refer to \cite{limin} for the details that motivating the use of indefinite proximal terms in the $2$-block majorized proximal ADMM, especially \cite[Section 6]{limin} on their computational merits, as well as \cite{zhangning} for the similar results in multi-block cases.}
linear operators $\cD_i$, $i=1,\ldots,s$, in the proximal terms, which unifies those used in Chen et al. \cite{chenl2015} and those used in Zhang et al. \cite{zhangning} to guarantee the viability of the block sGS decomposition techniques and the convergence of the whole sequence generated by the algorithm
in \eqref{iadmm}. Recall that the proximal terms in \cite{chenl2015} should be positive semidefinite while in \cite{zhangning} the functions being majorized should be separable with respect to each block of variables.
Here, we do not require $f$ to be separable and indefinite proximal terms are allowed.

\item
We use a unified criterion, which is weaker than those used in \cite{chenl2015}, for choosing the proximal terms in the algorithmic framework \eqref{iadmm} and analyzing its convergence.
Note that in \cite{chenl2015}, compared with the condition \cite[(3.2)]{chenl2015} imposed on choosing the proximal terms, a stronger condition (\cite[(5.26) of Theorem 5.1]{chenl2015}) was used to guarantee the convergence of the algorithm.
Here, we are able to get rid of such a gap while using a weaker condition.

\item
We conduct extensive numerical experiments on solving the linear and convex quadratic semidefinite programming (SDP) problems to demonstrate how the theoretical results obtained here can be exploited to improve the numerical efficiency of the implementation on ADMM.
Based on the numerical results, together with the theoretical analysis in this paper, we are able to give a
plausible explanation as to why ADMM often performs well when the dual step-length is chosen to be the golden ratio of $1.618$.
Meanwhile, a guiding principle on choosing the step-length during the practical implementation of the algorithmic framework in \eqref{iadmm} is derived.
\end{itemize}

\smallbreak
Here we emphasize again that for solving large-scale instances of the multi-block problem \eqref{probmulti},
a successful multi-block ADMM-type algorithm must not only possess convergence guarantee but should
also numerically perform at least as fast as the directly extended ADMM.
Based on our work in this paper, we can conclude that the inexact block sGS decomposition based majorized proximal ADMM studied in \cite{chenl2015} indeed does possess those desirable properties.
Moreover, this algorithm is a versatile framework and one can apply it to problem \eqref{probmulti} in different routines other than \eqref{iadmm}.
The reason that we are more interested in the iteration scheme \eqref{iadmm} is not only for the theoretical improvement one can achieve, but also for the practical merit it features for solving large scale problems, especially
when the dominating computational cost is in performing the evaluations associated with the linear mappings $\cG$ and $\cG^*$.
A particular case in point is the following problem:
\[
\label{probcqpie}
\min_{x\in\dX}\left\{\psi(x)+\frac{1}{2}\langle x,\cQ x\rangle-\langle c,x\rangle
\mid
\cG_E x=b_E,\ \cG_I x\ge {\bf b}_I \right\},
\]
where $\cQ$, $\psi$, and $c$ have the same meaning as in \eqref{modelexampled},
$\cG_E:\dX\to\dZ_E$ and $\cG_I:\dX\to\dZ_I$ are the given linear mappings,
and $b=({\bf b}_E;{\bf b}_I)\in \dZ:=\dZ_E\times\dZ_I$ is a given vector.
By introducing a slack variable $x'\in\dZ_I$, the above problem can be equivalently reformulated as
$$
\min_{x\in\dX,x'\in\dZ_I}\left\{\psi(x)
+\frac{1}{2}\langle x,\cQ x\rangle-\langle c,x\rangle
\mid
\begin{pmatrix}
\cG_E&0
\\
\cG_I&\cI_{\dZ_I}
\end{pmatrix}\begin{pmatrix}x\\x'\end{pmatrix}
=
b, \
x'\le 0
\right\},
$$
where $\cI_{\dZ_I}$ is the identity operator in $\dZ_I$.
The corresponding dual problem in the minimization form is then given by
$$
\min_{y_1, y_2', z}\left\{
p(y_1)
+\frac{1}{2} \langle y_2,\cQ y_2\rangle
-\langle b,z\rangle
\mid
\begin{pmatrix}
y_{11}
\\
y_{12}
\end{pmatrix}
+\begin{pmatrix}
\cQ \\
0
\end{pmatrix}
y_2
-
\begin{pmatrix}
\cG_E^*&\cG_I^*
\\
0&\cI_{\dZ_2}
\end{pmatrix}z
=\begin{pmatrix}
c
\\0
\end{pmatrix}
\right\},
$$
where $y_1:=(y_{11};y_{12})\in\dX\times \dZ_I$,
$p(y_1):=\psi_1^*(y_{11})+\delta_{+}(y_{12})$ with $\delta_+$ being the indicator function of the nonnegative orthant in $\dZ_I$,
$y_2\in\dX$ and $z\in\dZ$.
It is clear that when problem \eqref{probcqpie} has
a large number of inequality constraints, the dimension of $\dZ$ can be much larger than that of $\dX$.
For such a scenario, the iteration scheme \eqref{iadmm} is more preferable since the more difficult subproblem
involving $z$ is solved only once in each iteration.

\subsubsection*{Organization}
This paper is organized as follows.
In Section \ref{sec:exam}, we present a few important classes of problems that can be
handled by \eqref{probmulti} to illustrate the wide applicability of this model.
In Section \ref{sec:malm}, we design an {\em inexact} majorized proximal ALM framework and establish its
global and local convergence properties.
In Section \ref{sec:main}, we show the key result that the sequence generated by the inexact block sGS decomposition based majorized proximal ADMM \eqref{iadmm}, together with a simple error tolerance criterion, is equivalent to
the sequence generated by the inexact {ALM framework} introduced in Section \ref{sec:malm}.
Accordingly, the convergence of the two-block ADMM with the step-length in the interval of $(0,2)$
is also established for problem \eqref{probmulti} with $s=1$.
In Section \ref{sec:num}, we conduct extensive numerical experiments on the $2$-block dual linear SDP problems and the $multi$-block dual convex quadratic SDP problems to illustrate the numerical efficiency of the proposed algorithm, as well as the impact of the step-length on its numerical performance.
A few important practical observations from the numerical results are also presented. Finally, we conclude this paper in the last section.

\subsubsection*{Notation}
\begin{itemize}
\item
Let $\dH$ and $\dH'$ be two finite-dimensional real Hilbert spaces each endowed with an inner product $\langle\cdot,\cdot\rangle$ and its induced norm $\|\cdot\|$.
We also use $\|\cdot\|$ to denote the norm induced on the product space $\dH\times\dH'$ by the inner product $\langle (\nu_1,\nu_1'),(\nu_2,\nu_2')\rangle:=\langle \nu_1,\nu_2\rangle+\langle \nu_1',\nu_2'\rangle, \forall \nu_1,\nu_2\in\dH, \forall \nu_1',\nu_2'\in\dH'$.
\item
For any linear map $\cO:\dH\to\dH'$, we use $\cO^*$ to denote its adjoint,
$\cO^{-1}$ to denote its inverse (if invertible),
$\cO^\dag$ to denote its Moore--Penrose pseudoinverse, $\range(\cO)$ to denote its range space,
and $\|\cO\|$ to denote its spectral norm.

\item
If $\dH'=\dH$ and $\cO$ is self-adjoint and positive semidefinite, there must be a unique self-adjoint positive semidefinite operator, denoted by $\cO^{1/2}$, such that $\cO^{1/2}\cO^{1/2}=\cO$.
In this case, for any $\nu,\nu'\in\dH$ we define $\langle \nu,\nu'\rangle_{\cO}:=\langle \cO \nu,\nu'\rangle$ and $\|\nu\|_\cO:=\sqrt{\langle \nu, \cO \nu\rangle}=\|\cO^{1/2}\nu\|$.
If $\cO$ is also invertible, $\cO^{1/2}$ is invertible and we use the notation that $\cO^{-1/2}:=(\cO^{1/2})^{-1}$ .

\item
Let $\cO_1,\ldots,\cO_k$ be $k$ self-adjoint linear operators, we used $\diag(\cO_1,\ldots,\cO_k)$ to denote the block-diagonal linear operator whose block-diagonal elements are in the order of $\cO_1,\ldots,\cO_k$.

\item
For any convex set $H\subseteq \dH$, we denote the relative interior of $H$ by $\ri (H)$.
When the self-adjoint linear operator $\cO:\dH\to\dH$ positive definite, we define, for any $\nu\in\dH$,
$$\dist_\cO(\nu,H):=\inf_{\nu'\in H}\|\nu-\nu'\|_\cO
\quad\mbox{and}\quad
\Pi^\cO_H(\nu)=\argmin_{\nu'\in H}\|\nu-\nu'\|_\cO.$$
If $\cO$ is the identity operator we just omit it from the notation so that $\dist(\cdot, H)$ and $\Pi_H(\cdot)$ are the standard distance function and the metric projection operator, respectively.
\item
Let $\theta:\dH\to(-\infty,+\infty]$ be an arbitrary closed proper convex function.
We use $\dom\theta$ to denote its effective domain, $\partial\theta$ to denote its subdifferential mapping, and $\theta^*$ to denote its conjugate function.
Moreover, for a given self-adjoint and positive definite linear operator $\cO: \dH \to \dH$, we use $\prox_{\theta}^\cO$ to denote the Moreau-Yosida proximal mapping of $\theta$,
which is defined by
$$
\begin{array}{l}
\prox^\cO_\theta(\nu):=\argmin\limits_{\nu'\in\dH}\left\{\theta(\nu')
+\frac{1}{2}\|\nu-\nu'\|_\cO^2\right\},\quad \forall \nu\in\dH.
\end{array}
$$

Note that the mapping $\prox^{\cO}_{\theta}$ is globally Lipschitz continuous. If $\cO$ is the identity operator, we will drop $\cO$ from $\prox^\cO_\theta(\cdot)$.
\end{itemize}

\section{Illustrative Examples}
\label{sec:exam}
In this section, we present a few important classes of concrete problems, including those in the classic core convex programming as well as those which are popularly used in various real-world applications.
As will be shown, these problems and/or their dual problems
have the form given {by} \eqref{probmulti}, so that the algorithm designed in this paper can be utilized to solve them.

\subsection{Convex Composite Quadratic Programming }
\label{secccqp}

It is well known that
 many problems are subsumed under the convex composite quadratic programming model \eqref{modelexample} or the more concrete form \eqref{probcqpie}.
 For example, it includes the important classes of
convex quadratic programming (QP), the convex quadratic semidefinite programming (QSDP),
and the convex quadratic programming and weighted centering \cite{potra} (QPWC).
As an illustration, consider a convex QSDP problem {in} the following form
\[
\label{qsdp}
\min_{X\in\S^n}\left\{\frac{1}{2}\langle X,\mathbfcal{Q} X\rangle-\langle C,X\rangle\ \Big|\
\bA_E X={\bf b}_E,\ \bA_I X\ge {\bf b}_I, \ X\in\S_+^n \right\},
\]
where $\S^n$ is the space of $n\times n$ real symmetric matrices and $\S_+^n$ is
the closed convex cone of positive semidefinite matrices in $\S^n$,
$\mathbfcal{Q}:\S^n\to\S^n$ is a positive semidefinite linear operator,
$C\in\S^n$ is a given matrix, and
$\bA_E$ and $\bA_I$ are the linear maps from $\S^n$ to the two finite-dimensional Euclidean spaces $\R^{m_E}$ and $\R^{m_I}$ that containing ${\bf b}_E$ and ${\bf b}_I$, respectively.
To solve this problem, one may consult the recently developed software QSDPNAL in Li et al. \cite{liqsdpnal} and the references therein.
The algorithm implemented in QSDPNAL is a two-phase augmented Lagrangian method in which the first phase is an inexact sGS decomposition based multi-block proximal ADMM whose convergence was established in \cite[Theorem 5.1]{chenl2015}. The solution generated in the first phase is used as the initial point to warm-start
the second phase algorithm, which is an ALM with the inner subproblem in each iteration being solved via an inexact
semismooth Newton algorithm.
In Section \ref{sec:num}, we will use the QSDP problem \eqref{qsdp} to test the algorithm studied in this paper.

Besides the core optimization problems just mentioned above, there are many problems from real-word applications that
can be cast in the form of \eqref{modelexample} and the following are only a few such examples.

\subsubsection*{Penalized and Constrained Regression Models}

In various statistical applications,
the penalized and constrained (PAC) regression \cite{james} often arises in
high-dimensional generalized linear models with linear equality and inequality constraints.
A concrete example of the PAC regression is the following constrained lasso problem
\begin{equation}
\label{eq:classo}
\min_{x\in\R^n}
\left\{\frac{1}{2}\norm{\Phi x - \eta}^2 + \lambda \norm{x}_1
\mid
A_E x=b_E,\ A_I x\ge b_I \right\},
\end{equation}
where $\Phi\in\R^{m\times n}$, $A_E\in\R^{m_E\times n}$, $A_I\in\R^{m_I\times n}$, $\eta\in\R^m$, $b_E\in\R^{m_E}$ and $b_I\in\R^{m_I}$ are the given data, and
$\lambda > 0$ is a given regularization parameter.
The statistical properties of problem \eqref{eq:classo} have been studied in \cite{james}.
For more details on the applications of the model \eqref{eq:classo}, one may refer to \cite{james,gaines} and
the references therein.
In Gaines et al. \cite{gaines}, the authors considered solving \eqref{eq:classo} by first
reformulating it as a conventional QP via letting $x=x_+ - x_-$ and adding the extra constraints $x_+\ge 0$, $x_-\ge 0$, and then applying the primal ADMM to solve the conventional QP, in which all the subproblems should be solved exactly (or to very high accuracy) by iterative methods.
Such a combination may perform well for low dimensional problems with moderate sample sizes.
But for the more challenging and interesting high-dimensional cases where $n$ is extremely large and $m\ll n$,
the approach in \cite{gaines} is likely to face severe numerical difficulties
because of the presence of a huge number of constraints.
Fortunately,
the algorithm we designed in this paper can precisely handle those difficult cases
because the large linear systems associated with the huge number of constraints
are not required to solve to very high accuracy by an iterative solver.

\subsubsection*{Noisy Matrix Completion and Rank-Correction Step}
In Miao et al. \cite{miao}, the authors introduced a rank-correction step for matrix completion with fixed basis coefficients to overcome the shortcomings of the nuclear norm penalization model for such problems.
Let $\overline X\in{\mathds V}^{n_1\times n_2}$ (where ${\mathds V}^{n_1\times n_2}$ may represent the space of $n_1\times n_2$ real or complex matrices or the space of $n\times n$ real symmetric or Hermitian matrices) be the unknown true low-rank matrix and $\widetilde X_m$ is an initial estimator of $\overline X$ from the nuclear norm penalized least squares model. The rank-correction step is to solve the following convex optimization problem
\[
\label{rankcor}
\begin{array}{cl}
\ds
\min_{X}&
\frac{1}{2m}\|y-{\mathbfcal P}_o(X)\|^2+\rho_m\left(
\|X\|_*-\langle F(\widetilde X_m), X\rangle
\right)
\\[3mm]
\mbox{s.t.}&
{\mathbfcal P}_A(X)={\mathbfcal P}_A(\overline X),\
\|{\mathbfcal P}_B(X)\|_\infty\le b,
\end{array}
\]
where $y={\mathbfcal P}_o(\overline X)+\epsilon\in\R^m$ is the observed data for the matrix $\overline X$, ${\mathbfcal P}_o$ is the linear map corresponding to the observed entries, $\epsilon\in\R^m$ is the unknown error, $\rho_m>0$ is
a given penalty parameter,
and $F:{\mathds V}^{n_1\times n_2}\to {\mathds V}^{n_1\times n_2}$ is a spectral operator \cite{ding} whose precise definition can be found in \cite[Section 5]{miao}. Here the constraints ${\mathbfcal P}_A(X)={\mathbfcal P}_A(\overline X)$ and $\|{\mathbfcal P}_B(X)\|_\infty\le b$ represent the fixed elements and bounded elements of $X$, respectively.
 If $F$ and the equality constraints are vacuous, problem \eqref{rankcor} is exactly the noise matrix completion model considered in \cite{wainwright}, and a similar matrix completion model can be found in \cite{klopp}.
One may view \eqref{rankcor} as an instance of problem \eqref{modelexample}, and
whose corresponding linear operator $\cQ$ admits a very simple form.

\subsection{Two-Block Problems}
Next we present a few important classes of two-block problems whose objective functions contain a linear part.
\subsection*{Semidefinite Programming}
One of the most prominent examples of problem \eqref{probmulti} with 2 blocks of variables (i.e., $s=1$) is the
dual linear semidefinite programming (SDP) problem given by
\begin{eqnarray}
\label{probsdp}
\min_{Y,\, {\bf z}}\big\{\delta_{\S_+^n}(Y)-\langle {\bf b},{\bf z}\rangle\;|\;Y+\bA^*{\bf z}=C\big\},
\end{eqnarray}
where $\bA: \S^n\to \R^m$ is a given linear map, and ${\bf b}\in \R^m$ and $C\in \S^n$ are given data.
The notation $\delta_{\S_+^n}$ denotes the indicator function of $\S_+^n$.
For problem \eqref{probsdp}, various ADMM algorithms have been employed to solve the problem.
As far as we are aware of,
the classic two-block ADMM with unit step-length was first employed in Povh et al. \cite{povh}
under the name of boundary point method for solving
the SDP problem \eqref{probsdp}. It was later
extended in Malick et al. \cite{malick} with a convergence proof.
The ADMM approach was later used in the software SDPNAL developed by Zhao et al. \cite{zhao}
to warm-start a semismooth Newton method based ALM for solving problem \eqref{probsdp}.

In section \ref{sec:num}, we will conduct extensive numerical experiments on solving a few classes of linear SDP problems
via the two-block ADMM algorithm but with the dual step-length being chosen in the interval $(0,2)$, as it is guaranteed by this paper.

\subsubsection*{Equality Constrained Problems}
Consider the equality constrained problem
\[
\label{bp}
\min_{x\in\dX}\big\{\theta(x) \mid \cG x=b\big\},
\]
where
$\cG:\dX\to\R^m$ is a linear map, $b\in\R^m$ is a given vector, and
 $\theta:\dX\to(-\infty,+\infty]$ is a simple closed proper convex function such that
its proximal mapping can be computed efficiently.
The dual problem of \eqref{bp} can be written in the minimization form as
\[
\label{dbp}
\min_{y,z}\left\{
\theta^*( y)
-\langle b,z\rangle
\mid
 y-\cG^*z=0 \right\}.
\]
A concrete example of problem \eqref{bp}, with $\dX:=\R^n$
and $\theta(x):=\|x\|_1$, is the basis pursuit (BP) problem \cite{bpchen}, which has been wildly used in sparse signal recovery and image restoration. Another example of \eqref{bp} is
the nuclear norm based matrix completion problem
for which $\dX:=\R^{n_1\times n_2}$ and $\theta(x)=\|x\|_{*}$. Moreover, the so called tensor completion problem \cite{jiliu} also falls into this category.

We note that for the application problems just mentioned above, the dimension of $\dX$ is generally much larger than $m$, i.e., the dimension
of the linear constraints. Therefore from the computational viewpoint, it is generally more economical to apply
the two-block ADMM to the dual problem \eqref{dbp} instead of the
primal problem \eqref{bp} (by introducing an extra variable $x'$ and adding the condition $x-x'=0$)
 because the former will solve smaller $m\times m$ linear systems
in each iteration whereas the latter will correspondingly need to solve much larger linear systems.

\subsubsection*{Composite Problems}
A composite problem can take the following form
\[
\label{scp}
\min_{z\in\dZ}f\left(c-\cG^* z\right),
\]
where $f:\dZ\to(-\infty,+\infty]$ is a (possibly nonsmooth) closed proper convex function whose proximal mapping
can be computed efficiently, $\cG:\dZ\to\dX$ is a given linear operator and $c\in \dX$ is given data.
By introducing a slack variable, problem \eqref{scp} can be recast as
$$
\min_{y,z}\Big\{ f(y)\ |\ y+\cG^* z=c\Big\}.
$$
Problem \eqref{scp} contains many real-world applications such as the well-known least absolute deviation (LAD) problem (also known as
 least absolute error (LAE), least absolute value (LAV), least absolute residual (LAR), sum of absolute deviations, or the $\ell_1$-norm condition).
The model \eqref{scp} also includes the Huber fitting problem \cite{huber}.
We shall not continue with more examples
as there are too many applications to be listed here to serve as a literature review.

\subsubsection*{Consensus Optimization}
Consider the following problem
\[
\label{consensus}
\min\limits_{z\in\dZ} \left\{\sum_{i=1}^n f_i\left(\cG_i^* z\right)
\right\},
\]
where each $f_i$ is a closed proper convex function and each $\cG_i:\dY_i\to \dZ$ is a linear operator. The model \eqref{consensus} includes the global variable consensus optimization and general variable optimization, as well as their regularized versions (see \cite[Section 7]{boyd}), which have been well applied in many areas such as machine learning, signal processing and wireless communication \cite{boyd,bert89,teo,schizas,zhu2010}.
In the consensus optimization setting, it is usually preferable to solve subproblems each involving a subset
of the component functions $f_1,\ldots, f_n$ instead of all of them. Therefore, one can
 equivalently
recast problem \eqref{consensus} as
\[
\label{consp}
\min_{y,z}\left\{ \sum_{i=1}^n \, f_i(y_i) \mid y_i- \cG^*_i z=0, \ 1\le i\le n\right\}.
\]
Obviously, when applying the two-block ADMM to solve \eqref{consp}, the subproblem with respect to $y$ is separated into
$n$ independent problems that can be solved in parallel.
In \cite{boyd}, the variable $z$ in \eqref{consp} is called the central collector.
Besides, the network based decentralized and distributed computation of the consensus optimization, such as the distributed lasso in \cite{mateos}, also falls in the problems setting in this paper.

\section{An Inexact Majorized ALM with Indefinite Proximal Terms}
\label{sec:malm}

In this section, we present an inexact majorized indefinite-proximal ALM.
This algorithm, as well as its global and local convergence properties, not only constitutes a generalization of the original (proximal) ALM, but also paves the way for us to establish its equivalence relationship with the
inexact block sGS decomposition based indefinite-proximal multi-block ADMM in the next section.

Let $\dX$ and $\dW$ be two finite-dimensional real Hilbert spaces each endowed with an inner product $\langle\cdot,\cdot\rangle$ and its induced norm $\|\cdot\|$.
We consider the following fairly general linearly constrained convex optimization problem
\[
\label{prob1}
\min_{w\in\dW}
\Big\{
\varphi(w)+h(w)
\mid
\cA^* w=c
\Big\},
\]
where
$\varphi:\dW\to(-\infty,+\infty]$ is a closed proper convex function,
$h:\dW\to(-\infty,+\infty)$ is a continuously differentiable convex function whose gradient is Lipschitz continuous,
$\cA:\dX\to\dW$ is a linear mapping and $c\in\dX$ is the given data.
The Karush-Kuhn-Tucker (KKT) system of problem \eqref{prob1} is given by
\[
\label{kkt1}
0\in\partial\varphi(w)+\nabla h(w)+\cA x,
\quad
\quad \cA^* w-c=0.
\]
For any $(w,x)\in \dW\times\dX$ that solve the KKT system \eqref{kkt1}, $w$ is a solution to problem \eqref{prob1} while $x$ is a dual solution of \eqref{prob1}.

The fact that the gradient of $h$ is Lipschitz continuous implies that there exists a self-adjoint positive semidefinite linear operator ${\what\Sigma_h}:\dW\to\dW$, such that for any $w'\in\dW$, $h(w)\le \what h(w,w')$, where
\[
\label{majorize}
\what h(w,w'):=h(w')+\langle \nabla h(w'), w-w'\rangle+\frac{1}{2}\|w-w'\|^2_{\what\Sigma_h},\quad\forall w\in\dW.
\]
We call the function
$\what h(\cdot,w'):\dW\to(-\infty,+\infty)$ a majorization of $h$ at $w'$.
The following result, whose proof can be found in \cite[Lemma 3.2]{zhangning}, will be used later.
\begin{lemma}
\label{lemmazn}
Suppose that \eqref{majorize} holds for any given $w'\in\dW$. Then, it holds that
$$
\Big\langle \nabla h(w)-\nabla h(w'),
w''- w'
\Big\rangle
\ge-\frac{1}{4}\|w-w''\|^2_{\what\Sigma_h},
\quad\forall w,w',w''\in\dW.
$$
\end{lemma}

Let $\sigma>0$ be a given penalty parameter.
The majorized augmented Lagrangian function associated with problem \eqref{prob1} is defined by
\[
\label{alf1}
\begin{array}{r}
\cL_\sigma(w;(x,w')):=\varphi(w)+\what h(w,w')+\langle \cA^*w-c,x\rangle+\frac{\sigma}{2}\|\cA^*w-c\|^2,\quad
\\[3mm]
\forall(w,x,w')\in\dW\times\dX\times\dW.
\end{array}
\]
In the following, we propose an inexact majorized indefinite-proximal ALM in Algorithm \ref{alg:alm} for solving problem \eqref{prob1}.
This algorithm is an extension of the proximal method of multipliers developed by Rockafellar \cite{roc76b}, with
new ingredients added based on the recent progress on using proximal terms which are not necessarily positive definite \cite{fazel13,limin,zhangning}
and the implementable inexact minimization criteria studied in \cite{chenl2015}.
For the convenience of later convergence analysis, we make the following blanket assumption.
\begin{assumption}
\label{ass:basic}
The solution set to the KKT system \eqref{kkt1} is nonempty and $\cS:\dW\to\dW$ is a given self-adjoint (not necessarily positive semidefinite) linear operator such that
\[
\label{proxi}
\cS\succeq -\frac{1}{2}{\what\Sigma_h}
\quad\mbox{and}\quad
\frac{1}{2}{\what\Sigma_h}+\sigma\cA\cA^*+\cS\succ 0.
\]
\end{assumption}
We are now ready to present Algorithm \ref{alg:alm} that will be studied in this section.

%%%% Algorithm %%
\makeatletter
\renewcommand
\thealgorithm{{iPALM}}
\makeatother
\begin{algorithm}[H]
\caption{An inexact majorized indefinite-proximal ALM}
\label{alg:alm}
\normalsize
Let $\{\varepsilon_k\}$ be a summable sequence of nonnegative numbers.
Choose an initial point $(x^0,w^0)\in\dX\times\dW$.
For $k=0,1,\ldots$, perform the following steps in each iteration.
\begin{description}
\item[\bf Step 1.]
Compute
\[
\label{wkp1}
w^{k+1}\approx\overline w^{k+1}:=\argmin_{w\in\dW}\left\{\cL_\sigma\big(w;(x^k,w^k)\big)
+\frac{1}{2}\|w-w^k\|_{\cS}^2\right\}
\]
such that there exists a vector $d_k\in\dW$ satisfying
$\|d^k\|\le\varepsilon_k$
and
\[
\label{dk}
d^k\in\partial_w\cL_\sigma\big(w^{k+1};(x^k,w^k)\big)+\cS(w^{k+1}-w^k).
\]
\item[\bf Step 2.]
Compute $x^{k+1}:=x^k+\tau\sigma(\cA^*w^{k+1}-c)$ with $\tau\in(0,2)$ being the step-length.
\end{description}
\end{algorithm}

We shall next proceed to analyze the global convergence, the rate of local convergence and the iteration complexity of Algorithm \ref{alg:alm}. For notational convenience, we collect the total quadratic information in the objective function of \eqref{wkp1} as the following linear operator
\[
\label{defM}
\cM :={\what\Sigma_h}+\cS+\sigma\cA\cA^*.
\]
The following result presents two important inequalities for the subsequent analysis.
The first one characterizes the distance (with $\cM$ being involved in the metric) from the computed solution to the true solution of the subproblem in \eqref{wkp1}, while the second one presents a non-monotone descent property about the sequence generated by Algorithm \ref{alg:alm}.

\begin{proposition}
Suppose that Assumption \ref{ass:basic} holds.
Then,
\begin{itemize}
\item[\bf(a)]
the sequence $\{(x^k,w^k)\}$ generated by Algorithm \ref{alg:alm} and the auxiliary sequence $\{\overline w^k\}$ defined in \eqref{wkp1} are well-defined, and it holds that
\[\label{estd}
\|w^{k+1}-\overline w^{k+1}\|^2_{\cM}
\; \le\;
\langle d^k, w^{k+1}-\overline w^{k+1}\rangle;
\]
\item[\bf(b)]
for any given $(x^*,w^*)\in\dX\times\dW$ that solves the KKT system \eqref{kkt1} and $k\ge 1$, it holds that
\[
\label{ineq:mainb}
\begin{array}{l}
\ds
\left(\frac{1}{2\tau\sigma}\|x_e^{k+1}\|^2
+\frac{1}{2}\|w_e^{k+1}\|^2_{{\what\Sigma_h}+\cS}\right)
-\left(\frac{1}{2\tau\sigma}\|x_e^k\|^2
+\frac{1}{2}\|w_e^k\|^2_{{\what\Sigma_h}+\cS}\right)
\\[4mm]
\le
\ds
-\left(\frac{(2-\tau)\sigma}{2}\| \cA ^*w^{k+1}_e \|^2
+\frac{1}{2}\|w^{k+1}-w^k\|^2_{\frac{1}{2}{\what\Sigma_h}+\cS}
-\langle d^k, w^{k+1}_e\rangle
\right),
\end{array}
\]
where $x_e:=x-x^*$, $\forall x\in\dX$ and $w_e:=w-w^*$, $\forall w\in\dW$.
\end{itemize}
\end{proposition}
\begin{proof}
{\bf (a)}
From \eqref{proxi} and \eqref{defM} we know that
$\cM\succ 0$.
Hence, each of the subproblems in Algorithm \ref{alg:alm} is strongly convex so that each $\overline w^{k+1}$ is uniquely determined by $(x^k,w^k)$.
Note that, for the given $\varepsilon_k\ge 0$, one can always find a certain $w^{k+1}$ such that $\|d^k\|\le\varepsilon_k$ with $d^k$ being given in \eqref{dk}, see \cite[Lemma 4.5]{duthesis}.
Hence, Algorithm \ref{alg:alm} is well-defined.
According to \eqref{majorize} and \eqref{alf1},
the objective function in \eqref{wkp1} is given by
$$
\varphi(w)+\langle \nabla h(w^k)+\cA x^k,w\rangle
+\frac{\sigma}{2}\|\cA^*w-c\|^2
+\frac{1}{2}\|w-w^k\|_{{\what\Sigma_h}+\cS}^2,
$$
so that \eqref{dk} implies that
\[
\label{dknew}
 d^k\in\partial \varphi(w^{k+1})+ \nabla h(w^k)+\cA x^k
+\sigma \cA(\cA^*w^{k+1}-c)
+({\what\Sigma_h}+\cS)(w^{k+1}-w^k).
\]
Therefore, from the definitions of the Moreau-Yosida proximal mapping and $\cM$ in \eqref{defM}, one has that
$$
w^{k+1}=\prox_\varphi^\cM\left(\cM^{-1}\big[d^k-(\nabla h(w^k)+\cA x^k-{\sigma}\cA c-({\what\Sigma_h}+\cS)w^k)\big]\right).
$$
Consequently, by the Lipschitz continuity of $\prox_{\varphi}^{\cM}$ \cite[Proposition 2.3]{lemarechal} and the fact that $d^k$ can be set as zero if $w^{k+1}=\overline w^{k+1}$, one can readily get \eqref{estd}.

\smallskip
\noindent{\bf (b)}
Let $(x^*,w^*)\in\dX\times\dW$ be an arbitrary solution to the KKT system \eqref{kkt1}. Obviously, one has that
$-\nabla h(w^*)-\cA x^*\in\partial\varphi(w^*)$ and $\cA^* w^* = c$. This, together with \eqref{dknew} and the maximal monotonicity of $\partial\varphi$, implies that
$$
\langle d^k- \nabla h(w^k)+\nabla h(w^*)
-\cA x_e^k
-\sigma \cA(\cA^*w^{k+1}-c)
-({\what\Sigma_h}+\cS)(w^{k+1}-w^k)
, w^{k+1}_e
\rangle
\ge 0.
$$
Therefore, by using the fact that
\[
\label{eq:wkxk}
\cA^*w^{k+1}-c=\cA^*w_e^{k+1}=\frac{1}{\tau\sigma}(x^{k+1}-x^k),\] one can obtain from the above inequality and Lemma \ref{lemmazn} that
\[
\label{keyineq1}
\begin{array}{l}
\left\langle d^k, w^{k+1}_e \right\rangle
-\frac{1}{\tau\sigma}\left\langle x_e^k, x^{k+1}-x^k\right\rangle
-\sigma \|\cA^*w_e^{k+1}\|^2
-\left\langle ({\what\Sigma_h}+\cS)(w^{k+1}-w^k)
, w^{k+1}_e \right\rangle
\\[4mm]
\ge \left\langle \nabla h(w^k) - \nabla h(w^*), w^{k+1}_e
\right\rangle
\ge
-\frac{1}{4}\|w^{k+1}-w^k\|_{\what\Sigma_h}^2
=
-\frac{1}{2}\|w^{k+1}-w^k\|_{\frac{1}{2}\what\Sigma_h}^2.
\end{array}
\]
Note that
$\langle x_e^k, x^{k+1}-x^k\rangle
=\frac{1}{2}\|x_e^k\|^2-\frac{1}{2}\|x_e^k\|^2-
\frac{1}{2}\|x^{k+1}-x^k\|^2$
and
$$
\left\langle ({\what\Sigma_h}+\cS)(w^{k+1}-w^k), w^{k+1}_e\right\rangle
=\frac{1}{2}\|w^{k+1}-w^k\|^2_{{\what\Sigma_h}+\cS}
+\frac{1}{2}\|w_e^{k+1}\|^2_{{\what\Sigma_h}+\cS}
-\frac{1}{2}\|w_e^k\|^2_{{\what\Sigma_h}+\cS}.
$$
Then, \eqref{ineq:mainb} follows form \eqref{keyineq1} and this completes the proof of the proposition.
\qed
\end{proof}

\subsection{Global Convergence}
For the convenience of our analysis, we define the following two linear operators
\[
\label{defXT}
\Xi := \tau\sigma\left({\frac{1}{2}{\what\Sigma_h}+\cS+\frac{(2-\tau)\sigma}{6}\cA\cA^*}\right)
\quad\mbox{and}\quad
\Theta := \tau\sigma\left(\,{\what\Sigma_h}+\cS+\frac{(2-\tau)\sigma}{3}\cA\cA^*\right),
\]
which will be used in defining metrics in $\dW$.
Note that $\tau\in(0,2)$. If \eqref{proxi} in Assumption \ref{ass:basic} holds, one has that
\[
\label{eq-MM}
\left\{
\begin{array}{ll}
\frac{1}{\tau\sigma}\Xi&=
\frac{2-\tau}{6}\left(\frac{1}{2}{\what\Sigma_h}+\cS+\sigma\cA\cA^*\right)
+\left(1-\frac{2-\tau}{6}\right)
\Big(\frac{1}{2}{\what\Sigma_h}+\cS\Big)
\\[4mm]
&\succeq \frac{2-\tau}{6}\Big(\frac{1}{2}{\what\Sigma_h}+\cS+\sigma\cA\cA^*\Big)\succ 0,
\\[4mm]
\frac{1}{\tau\sigma}
\Theta&=\frac{1}{\tau\sigma}\Xi+\frac{1}{2}{\what\Sigma_h}+\frac{(2-\tau)\sigma}{6}\cA\cA^*\succeq \frac{1}{\tau\sigma}\Xi\succ 0.
\end{array}
\right.
\]
Moreover, we define the block-diagonal linear operator
$ \Omega:\dU\to\dU$ by
\[
\label{omega}
\Omega(x;w):=\left(x; \Theta^{\frac{1}{2}} w\right),\quad\forall (x,w)\in\dX\times\dW,
\]
where $\Theta$ is given by \eqref{defXT}.
Now we establish the convergence theorem of
Algorithm \ref{alg:alm}.
The corresponding proof mainly follows from the proof of \cite[Theorem 5.1]{chenl2015} for the convergence of an inexact majorized semi-proximal ADMM and the following result on quasi-Fej\'er monotone sequence will be used.

\begin{lemma}
\label{lemma:sq-sum}
Let $\{a_k\}_{k\ge0}$ be a sequence of nonnegative real numbers sequence satisfying $a_{k+1}\le a_k+\varepsilon_k$ for all $k\ge 0$, where $\{\varepsilon_k\}_{k\ge 0}$ is a nonnegative and summable sequence of real numbers. Then the $\{a_k\}$ converges to a unique limit point.
\end{lemma}

\begin{theorem}
\label{thm:main}
Suppose that Assumption \ref{ass:basic} holds and the sequence $\{(x^k,w^k)\}$ is generated by Algorithm \ref{alg:alm}. Then,
\begin{itemize}
\item[\bf(a)]
for any solution $(x^*,w^*)\in\dX\times\dW$ of the KKT system \eqref{kkt1} and $k\ge 1$, we have that
\[
\label{result}
\begin{array}{l}
\ds
\left\|(x_e^{k+1};w_e^{k+1})\right\|^2_{\Omega}
-\left \|(x_e^k; w_e^k)\right\|^2_{\Omega}
\\[3mm]\ds
\le
-\left(\frac{(2-\tau)}{3\tau}\| x^{k+1}-x^k \|^2
+\|w^{k+1}-w^k\|^2_{\Xi}
-2\tau\sigma\langle d^k, w^{k+1}_e\rangle
\right),
\end{array}
\]
where $x_e:=x-x^*$, $\forall x\in\dX$ and $w_e:=w-w^*$, $\forall w\in\dW$;

\item[\bf (b)] the sequence $\{(x^k,w^k)\}$ is bounded;
\item[\bf (c)] any accumulation point of the sequence $\{(x^k,w^k)\}$ solves the KKT system \eqref{kkt1};
\item[\bf (d)] the whole sequence $\{(x^k,w^k)\}$ converges to a solution to the KKT system \eqref{kkt1}.
\end{itemize}

\end{theorem}
\begin{proof}
{\bf(a)}
By using \eqref{defXT}, together with the definitions of $\Xi$ and $\Theta$ in \eqref{ineq:mainb},
and the fact that $\cA^*w_e^{k+1}=\frac{1}{\tau\sigma}(x^{k+1}-x^k)$, one can get
$$
\begin{array}{l}
\left\|(x_e^{k+1};w_e^{k+1})\right\|^2_{\Omega}
-\left \|(x_e^k; w_e^k)\right\|^2_{\Omega}
=
\left(\|x_e^{k+1}\|^2
+\|w_e^{k+1}\|^2_{\Theta}
\right)
-\left(\|x_e^k\|^2
+\|w_e^k\|^2_{\Theta}
\right)
\\[3mm]
\le
-\tau\sigma
\left(\frac{2(2-\tau)\sigma}{3}\| \cA ^*w^{k+1}_e \|^2
+\frac{(2-\tau)\sigma}{3}\| \cA ^*w^{k}_e \|^2
+\|w^{k+1}-w^k\|^2_{\frac{1}{2}{\what\Sigma_h}+\cS}
-2\langle d^k, w^{k+1}_e\rangle
\right)
\\[3mm]
\le
-\Big(\frac{(2-\tau)}{3\tau}\| x^{k+1}-x^k \|^2
+\tau\sigma\frac{(2-\tau)}{3}
\big(\sigma\|\cA^*w^{k+1}_e \|^2
+\sigma\|\cA^*w^{k}_e \|^2\big)
\\[2mm]
\qquad\quad
+\tau\sigma\|w^{k+1}-w^k\|^2_{\frac{1}{2}{\what\Sigma_h}+\cS}-2\tau\sigma\langle d^k, w^{k+1}_e\rangle
\Big)
\\[3mm]
\le
-\Big(\frac{(2-\tau)}{3\tau}\| x^{k+1}-x^k \|^2
+\sigma\tau\| w^{k}- w^{k+1}\|_{\frac{(2-\tau)\sigma}{6}\cA\cA^*+\frac{1}{2}{\what\Sigma_h}+\cS}^2-2\tau\sigma\langle d^k, w^{k+1}_e\rangle
\Big),
\end{array}
$$
which, together with the definition of the linear operator $\Xi$ in \eqref{defXT}, implies \eqref{result}.

\smallskip
\noindent{\bf(b)}
Define $\overline x^{k+1}:=x^k+\tau\sigma(\cA^*\overline w^{k+1}-c)$, $\forall k\ge 0$.
From \eqref{wkp1}, \eqref{dk} and \eqref{result} one can get that for any $k\ge 0$,
$$
\left\|(\overline x_e^{k+1};\overline w_e^{k+1})\right\|^2_{\Omega}
\le\left \|(x_e^k; w_e^k)\right\|^2_{\Omega}
-\left(
\frac{(2-\tau)}{3\tau}\| \overline x^{k+1}-x^k \|^2
+\|\overline w^{k+1}-w^k\|^2_{\Xi}
\right).
$$
Meanwhile, one can get that
$\|(\overline x_e^{k+1};\overline w_e^{k+1}) \|_{\Omega}
\le \|(x_e^k; w_e^k) \|_{\Omega},\,\forall\,k\ge 0$.
Therefore, it holds that
\[
\label{ineqpf1}
\begin{array}{ll}
\big\|( x_e^{k+1}; w_e^{k+1}) \big\|_{\Omega}
\le
\big\|(\overline x_e^{k+1};\overline w_e^{k+1})\big\|_{\Omega}
+\big\|\big(x^{k+1}-\overline x^{k+1};
w^{k+1}-\overline w^{k+1}\big)\big\|_\Omega
\\[3mm]
\le\big\|(x_e^k; w_e^k)\big\|_{\Omega}
+\big\|\big(
\tau\sigma\cA^*(w^{k+1}-\overline w^{k+1});
\Theta^{1/2}( w^{k+1}-\overline w^{k+1})\big)\big\|
\\[3mm]
=\big\|(x_e^k; w_e^k)\big\|_{\Omega}
+\big\|w^{k+1}-\overline w^{k+1}\big\|_{\tau^2\sigma^2\cA\cA^*+\Theta} \,,\quad\forall k\ge 0.
\end{array}
\]
From \eqref{estd} we know that
$\| w^{k+1}-\overline w^{k+1}\|_{\cM}^2
\le
\langle \cM^{-1/2}d^k, \cM^{1/2}(w^{k+1}-\overline w^{k+1})\rangle$,
so that
$$
\|\cM^{1/2}(w^{k+1}-\overline w^{k+1})\|
\le
\|\cM^{-1/2}d^k\|
\le \|\cM^{-1/2}\|\|d^k\|.
$$
Therefore, it holds that
 \[
\label{ineq:Hx}
\begin{array}{ll}
\|w^{k+1}-\overline w^{k+1}\|
&\le
\|\cM^{-1/2}\|\|\cM^{1/2}(w^{k+1}-\overline w^{k+1})\|
\\[3mm]
&\le \|\cM^{-1/2}\|\|\cM^{-1/2}\|\|d^k\|
\le
\|\cM^{-1}\|\varepsilon_k,\quad\forall k\ge 0.
\end{array}
\]
Therefore, by combining \eqref{ineqpf1} and \eqref{ineq:Hx} together we can get
$$
\big\|( x_e^{k+1}; w_e^{k+1}) \big\|_{\Omega}
\le
\big\|(x_e^{k}; w_e^{k}) \big\|_{\Omega}
+\sqrt{\|\tau^2\sigma^2\cA\cA^*+\Theta\|}\,\|\cM^{-1}\|\varepsilon_k,\quad\forall k\ge 0.
$$
Hence, the sequence $\left\{\|(x_e^{k+1}; w^{k+1}_e)\|_\Omega\right\}$ is quasi-Fej\'er monotone, which converges to a unique limit point by Lemma \ref{lemma:sq-sum}. Since $\Omega$ defined in \eqref{omega} is positive definite, we further know that the sequence $\{(x^k,w^k)\}$ is bounded.

\smallskip
\noindent{\bf(c)}
From \eqref{result} we know that for any $k\ge 0$,
\[
\label{eq:xiwi}
\begin{array}{l}
\ds
\sum_{i=0}^k\left\{
\left \|(x_e^i; w_e^i)\right\|^2_{\Omega}
-\left\|(x_e^{i+1};w_e^{i+1})\right\|^2_{\Omega}
+2\tau\sigma\varepsilon_i\|w^{i+1}_e\|
\right\}
\\[3mm]\ds
\ge
\sum_{i=0}^k\left\{
\frac{(2-\tau)}{3\tau}\| x^{i+1}-x^i \|^2
+\|w^{i+1}-w^i\|^2_{\Xi}
\right\}.
\end{array}
\]
Since $\{(x^k,w^k)\}$ is bounded and $\{\varepsilon_k\}$ is summable, it holds that
$$
\sum_{i=0}^\infty \left\{
\left \|(x_e^i; w_e^i)\right\|^2_{\Omega}
-\left\|(x_e^{i+1};w_e^{i+1})\right\|^2_{\Omega}
+2\tau\sigma\varepsilon_i\|w^{i+1}_e\|
\right\} < \infty,
$$
which, together with \eqref{eq:xiwi} and the fact that $\Xi \succ 0$, implies
\[
\label{limitzerox}
\lim_{k\to\infty}(x^{k+1}-x^k)=0
\quad\mbox{and}\quad
\lim_{k\to\infty}(w^{k+1}-w^k)=0.
\]
Suppose that the subsequence $\{(x^{k_j},w^{k_j})\}$ of $\{(x^k,w^k)\}$ converges to some limit point $(x^\infty,w^{\infty})$.
By taking limits on both sides of
\eqref{dknew} and \eqref{eq:wkxk} along with $k^j$
and using \eqref{limitzerox} and \cite[Theorem 24.6]{rocbook}, one can get
$$
0\in\partial \varphi(w^{\infty})+ \nabla h(w^\infty)+\cA x^\infty \quad \mbox{and} \quad \cA^*w^{\infty}-c=0,
$$
which implies that $(x^{\infty},w^{\infty})$ is a solution to the KKT system \eqref{kkt1}.

\smallskip
\noindent{\bf(d)} Note that \eqref{result} holds for any $(x^*,w^*)$ satisfying the KKT system \eqref{kkt1}. Therefore, we can choose $x^* = x^\infty$ and $w^* = w^\infty$ in \eqref{result}:
$$
\begin{array}{l}
\|( x^{k+1}-x^{\infty};w^{k+1}-w^{\infty})\|^2_{\Omega}
\le\|(x^k-x^{\infty};w^k-w^\infty)\|^2_{\Omega}
+
2\tau\sigma\|w^{k+1}-w^\infty\|\varepsilon_k.
\end{array}
$$
Note that $\{w^k\}$ is bounded. Then, the above inequality, together with Lemma \ref{lemma:sq-sum}, implies that the quasi-Fej\'er monotone sequence $\{ \|(x^k-x^{\infty};w^k-w^\infty)\|^2_{\Omega}\}$ converges. Since $(x^{\infty},w^{\infty})$ is a limit point of $\{(x^k,w^k)\}$, one has that
$$
\lim_{k\to 0} \|(x^k-x^{\infty};w^k-w^\infty)\|^2_{\Omega} =0,$$
 which, together with the fact that $\Omega \succ 0$, implies that the whole sequence $\{(x^k,w^k)\}$ converges to $(x^{\infty},w^{\infty})$. This completes the proof of the theorem.
\qed
\end{proof}

\subsection{Local Convergence Rate}
In this section, we present the local convergence rate analysis of Algorithm \ref{alg:alm}.
For this purpose, we denote $\dU:=\dX\times\dW$ and consider the KKT residual mapping of problem \eqref{prob1} defined by
\[
\label{residual}
\cR(u)=\cR(x,w):=\begin{pmatrix}
c-\cA^* w
\\[2mm]
w-\prox_{\varphi}(w-\nabla h(w)-\cA x)
\end{pmatrix},
\quad
\forall u=(x,w)\in\dX \times\dW.
\]
Note that $\cR(u)=0$ if and only if $u=(x,w)$ is a solution to the KKT system \eqref{kkt1}, whose solution set can therefore be characterized by $\bK:=\{u\,|\, \cR(u)=0\}$.
Moreover, the residual mapping $\cR$ has the following property.
\begin{lemma}
\label{lem;convineq}
Suppose that Assumption \ref{ass:basic} holds and the sequence $\{u^k:=(x^k,w^k)\}$ is generated by Algorithm \ref{alg:alm}.
Then, for any $k\ge 0$,
\[
\label{ineqres}
\begin{array}{ll}
\ds
\|\cR(u^{k+1})\|^2
\le&
\ds
\frac{1}{\tau^2\sigma^2}\|x^{k}-x^{k+1}\|^2
+\frac{2\|{\what\Sigma_h}+\cS\|}{\tau\sigma}
\|w^{k+1}-w^k\|_{\Theta}^2
\\[3mm]
&
\ds+2\|(1-\tau^{-1})\cA(x^{k+1}-x^k)+d^k\|^2.
\end{array}
\]
\end{lemma}

\begin{proof}
Note that \eqref{dknew} holds. Then, one can see that
$$
w^{k+1}=\prox_\varphi
\left(w^{k+1} + d^k
-\nabla h(w^k)
-\cA \left(x^k+\frac{x^{k+1}-x^k}{\tau}\right)
-\big(\what\Sigma_h+\cS\big)(w^{k+1}-w^k)\right).
$$
By taking the above equality and $c-\cA^*w^{k+1}=\frac{1}{\tau\sigma}(x^{k}-x^{k+1})$ into the definition of $\cR(u^{k+1})$ in \eqref{residual} and using the Lipschitz continuity of $\prox_\varphi$, one can get
\[
\label{ineq321}
\begin{array}{ll}
\ds
\|\cR(u^{k+1})\|^2
\le&
\ds
\frac{1}{\tau^2\sigma^2}\|x^{k}-x^{k+1}\|^2
+2\|(1-\tau^{-1})\cA(x^{k+1}-x^k)+d^k\|^2
\\[3mm]
&
\ds
+2\|\nabla h(w^{k+1})-\nabla h(w^k)
-\big(\what\Sigma_h+\cS\big)(w^{k+1}-w^k)
\|^2.
\end{array}
\]
By using Clarke's mean value theorem \cite[Proposition 2.6.5]{clarke} we know that for any $k\ge 0$ there exists a linear operator $\Sigma_k:\dW\to\dW$ such that
$\nabla h(w^{k+1})-\nabla h(w^k)=\Sigma_k(w^{k+1}-w^k)$
with $0\preceq\Sigma_k\preceq\what\Sigma_h$,
so that
\[
\label{ineq322}
\begin{array}{ll}
\|
\nabla h(w^{k+1})-\nabla h(w^k)
-\big(\what\Sigma_h+\cS\big)(w^{k+1}-w^k)
\|^2
\\[3mm]
=
\|
\big(\what\Sigma_h+\cS-\Sigma_k\big)(w^{k+1}-w^k)
\|^2
\\[3mm]
\le \|{\what\Sigma_h}+\cS-\Sigma_k\|
\langle w^{k+1}-w^k , \big(\what\Sigma_h+\cS-\Sigma_k\big)(w^{k+1}-w^k)\rangle
\\[3mm]
\le
\|{\what\Sigma_h}+\cS\|
\langle w^{k+1}-w^k , \big(\what\Sigma_h+\cS\big)(w^{k+1}-w^k)\rangle
\\[3mm]
\le \frac{\|{\what\Sigma_h}+\cS\|}{\tau\sigma}
\left\langle w^{k+1}-w^k , \tau\sigma\Big(\what\Sigma_h+\cS+\frac{(2-\tau)\sigma}{3}\cA\cA^*\Big)(w^{k+1}-w^k)\right\rangle,\quad\forall k\ge 0,
\end{array}
\]
where the last inequality comes form the fact that $0<\tau<2$. Then, by using the definition of $\Theta$ in \eqref{defXT}, one can readily see from \eqref{ineq321} and \eqref{ineq322} that
\eqref{ineqres} holds. This completes the proof.
\qed\end{proof}
To analyze the linear convergence rate of Algorithm \ref{alg:alm}, we shall introduce the following error bound condition.
\begin{definition}
The KKT residual mapping $\cR$ defined in \eqref{residual} is said to be metric subregular\footnote{
This is equivalent to say that $\cR^{-1}$ is calm at $0\in\dU$ for $\overline u\in \bK$ with the same modulus $\kappa>0$, see \cite[Theorem 3H.3]{dontchev}.} \cite[\mbox{3.8 [3H]}]{dontchev} (with the modulus $\kappa>0$) at $\overline u\in \bK$ for $0\in\dU$ if there exists a constant $r>0$ such that
\[
\label{msr}
\dist\big(u, \bK\big)\le\kappa\|\cR(u)\|\, ,\quad\forall u\in\{u\in\dU\mid\|u-\overline u\|\le r\}.
\]
\end{definition}

Suppose that Assumption \ref{ass:basic} holds. We know from
\eqref{defXT} and \eqref{eq-MM} that $\Xi\succ 0$.
Hence, one can let $\zeta>0$ be the smallest real number such that $\zeta\Xi\succeq\Theta$.
For notational convenience, we define the following positive constants:
\begin{numcases}{}
\rho:=\max\left\{
\frac{6\sigma^2 (\tau-1)^2\|\cA^*\cA\|+3}
{\tau\sigma^2(2-\tau)}\, ,
\frac{2\zeta\|{\what\Sigma_h}+\cS\|
}{\tau\sigma}
\right\}
\max\Big\{\|\Theta\|,1\Big\},
\label{defrho}
\\
\beta:=\max\left\{{\sqrt{\zeta}},\sqrt{{3\tau }/(2-\tau)}\right\},
\label{betamubeta}
\\
\mu:=\sqrt{\tau\sigma\|{\what\Sigma_h}+\cS
+\frac{2}{3}(1+\tau)\sigma\cA\cA^*\|}\; \|\cM^{-1}\|.
\label{betamumu}
\end{numcases}
To ensure the local linear rate convergence of Algorithm \ref{alg:alm}, we need extra conditions to control the error variable $d^k$ in each iteration. Hence, we make the following assumption.
\begin{assumption}
\label{assrate}
There exists an integer $k_0>0$ and
a sequence of nonnegative real numbers
$\{\eta_k\}$ such that
\[
\label{dkerr}
\begin{array}{ll}
\sup\limits_{k\ge k_0} \{\eta_k\}< 1/\mu
\quad \mbox{and}\quad
\|d^k\|\le \eta_k\|u^k-u^{k+1}\|,
\quad
\forall\, k\ge k_0.
\end{array}
\]
\end{assumption}

Now we are ready to present the local convergence rate of Algorithm \ref{alg:alm}.
\begin{theorem}
\label{ratethm}
Suppose that Assumptions \ref{ass:basic} and \ref{assrate} hold. Let $\{u^k=(x^k,w^k)\}$ be the sequence generated by Algorithm \ref{alg:alm} that converges to $u^*:=(x^*,w^*)\in \bK$.
Suppose that the KKT residual mapping $\cR$ defined in \eqref{residual} is metric subregular at $u^*$ for $0\in\dU$ with the modulus $\kappa>0$, in the sense that
there exists a constant $r>0$ such that \eqref{msr} holds with $\overline u=u^*$.
Then, there exists a threshold $\bar k>0$ such that for for all $k\ge\bar k$,
\[
\label{rateresult}
\dist_\Omega\Big(u^{k+1},\bK\Big)
\le\vartheta^k\,
\dist_\Omega\Big(u^k,\bK\Big)
\quad\mbox{with}
\quad
\vartheta^k:=\frac{1}{1-\mu\eta_k}
\left(\sqrt{\frac{\kappa^2\rho}{1+\kappa^2\rho}}
+\mu\eta_k(1+\beta)\right).
\]
Moreover, if it holds that
\[
\label{etarate}
\sup\limits_{k\geq \bar k}
\{\eta_k\}
<\frac{1}{\mu(2+\beta)}\left(1- \sqrt{\frac{\kappa^2\rho}{1+\kappa^2\rho}}\,\right),
\]
then one has
$\sup_{k\geq \bar k}\{\vartheta^k\}<1$,
and the convergence rate of $\dist_\Omega(u^k,\bK)$ is Q-linear when $k\ge\overline k$.
\end{theorem}

\begin{proof}
Denote $u_e:=u-u^*$ for all $u\in\dU$ and define
$
\overline x^{k+1}:=x^k+\tau\sigma(\cA^* \overline w^{k+1}-c)$ and $\overline u^{k+1}:=(\overline x^{k+1},\overline w^{k+1})$, $\forall k\ge 0$.
Since $\{u^k\}$ converges to $u^*$ and $\{d^k\}$ converges to $0$ as $k\to\infty$,
one has from \eqref{estd} that $\{\overline u^k\}$ also converges to $u^*$ as $k\to\infty$.
Therefore, there exists a threshold $\bar k>0$ such that
\[
\label{resi}
\|\overline u_e^{k+1}\|\le r
\quad\mbox{and}\quad
\|u_e^{k+1}\|\le r\, ,\quad\forall k\ge \bar k.
\]
According to Lemma \ref{lem;convineq}, one can let $u^{k+1}=\overline u^{k+1}$ and $d^k=0$ in \eqref{ineqres} and use the fact that $\zeta\Xi\succeq\Theta$ to obtain that
\[
\label{ineqres2}
\begin{array}{l}
\ds
\|\cR(\overline u^{k+1})\|^2
\\[3mm]
\ds
\le
\left(
\frac{2{(\tau-1)^2}\|\cA^*\cA\|}{\tau^2}+\frac{1}{(\tau\sigma)^{2}}\right)
\|\overline x^{k+1}- x^k\|^2
+\frac{2\|{\what\Sigma_h}+\cS\|
}{\tau\sigma}\| \overline w^{k+1}-w^k\|_{\Theta}^2
\\[3mm]
\ds
\le
\max\left\{
\frac{{6\sigma^2 (\tau-1)^2\|\cA^*\cA\|+3}}
{\tau\sigma^2(2-\tau)}
\ ,
\frac{2\zeta\|{\what\Sigma_h}+\cS\|
}{\tau\sigma}
\right\}
\left(
\frac{(2-\tau)}{3\tau}
\|\overline x^{k+1}- x^k\|^2
+\|\overline w^{k+1}-w^k\|_{\Xi}^2
\right).
\end{array}
\]
Moreover, according to the definition of $\Omega$ in \eqref{omega}, one has that
$$
\dist_\Omega^2\big(u,\bK\big)
\le \max\{\|\Theta\|,1\}
\dist^2\big(u,\bK\big),\quad \forall\, u\in\dU.
$$
Then, by using the above inequality, together with \eqref{msr}, \eqref{resi} and \eqref{ineqres2}, we can obtain with the constant $\rho>0$ being defined in \eqref{defrho} that
\[
\label{fie1}
\begin{array}{rl}
\ds
\dist^2_\Omega(\overline u^{k+1},\bK)
&
\ds\le\kappa^2\max\{\|\Theta\|,1\}\|\cR(\overline u^{k+1})\|^2
\\[3mm]
&
\ds
 \le\kappa^2\rho
\left(
\frac{(2-\tau)}{3\tau}
\|\overline x^{k+1}- x^k\|^2
+\|\overline w^{k+1}-w^k\|_{\Xi}^2
\right),
\quad
\forall\,  k\ge \bar k.
\end{array}
\]
It is easy to see from \eqref{result} that, for any $k\ge 0$,
\[
\label{ineqge}
\dist_\Omega^2(\overline u^{k+1},\bK)\le
\dist_\Omega^2(u^k,\bK)
-\left(\frac{(2-\tau)}{3\tau}\|\overline x^{k+1}-x^k \|^2
+\|\overline w^{k+1}-w^k\|^2_{\Xi}
\right).
\]
Therefore, by combining \eqref{fie1} and \eqref{ineqge} together we can get
\[
\label{ineqmm1}
 \dist^2_\Omega(\overline u^{k+1},\bK)
\le\frac{\kappa^2\rho}{1+\kappa^2\rho}
\dist_\Omega^2(u^k,\bK),
\quad\forall\, k\ge \bar k.
\]
From \eqref{ineqge} and the fact that
$\zeta\Xi\succeq\Theta$
we know that
$$
\dist_\Omega^2(u^k,\bK)
\ge\min\left\{\frac{(2-\tau)}{3\tau},\frac{1}{\zeta}\right\} \|u^{
k}-\overline u^{k+1}\|_{\Omega}^{2},
\quad\forall\,k\ge 0.
$$
Therefore, it holds that
\[
\label{tempineq1}
\|u^{k}-\overline u^{k+1}\|_\Omega\le \beta\,\dist_\Omega(u^k,\bK),\quad \forall\, k\ge 0,
\]
where the constant $\beta>0$ is given in \eqref{betamubeta}.
By using the triangle inequality, we have that
\[
\label{trian}
\|u^k-\Pi^\Omega_{\bK}(\overline u^{k+1})\|_\Omega
\le \dist_\Omega(u^k,\bK)
+\|\Pi^\Omega_{\bK}(u^k)-\Pi^\Omega_{\bK}( \overline u^{k+1})\|_\Omega,\quad\forall k\ge 0.
\]
Moreover, from \cite[Proposition 2.3]{lemarechal}
we know that
$$
\|\Pi^\Omega_{\bK}(u^k)-\Pi^\Omega_{\bK}( \overline u^{k+1})\|_\Omega^2\le \langle
\Pi^\Omega_{\bK}(u^k)-\Pi^\Omega_{\bK}( \overline u^{k+1}),
\Omega(u^k- \overline u^{k+1})
\rangle,
\quad\forall\, k\ge 0.
$$
Thus, one has
$$\|\Pi^\Omega_{\bK}(u^k)-\Pi^\Omega_{\bK}( \overline u^{k+1})\|_\Omega\le \| u^k- \overline u^{k+1}\|_\Omega,\quad\forall k\ge 0,$$
which together with \eqref{tempineq1} and \eqref{trian}, implies that
\[
\label{eep2}
\|u^k-\Pi^\Omega_{\bK}(\overline u^{k+1})\|_\Omega
\le(1+\beta)\dist_\Omega(u^k,\bK),
\quad\forall\,k\ge 0.
\]
From the definitions of $\Theta$ in \eqref{defXT} and
$\Omega$ in \eqref{omega} we know that
$$
\begin{array}{l}
\ds
\|u^{k+1}-\overline u^{k+1}\|^2_\Omega \;=\;
(\tau\sigma)^2\|\cA^*(w^{k+1}-\overline w^{k+1})\|^2
+\|w^{k+1}-\overline w^{k+1}\|^2_{\Theta}
\\[3mm]
=
\left\langle w^{k+1}-\overline w^{k+1},
\tau\sigma({\what\Sigma_h}+\cS+\frac{2}{3}(1+\tau)\sigma\cA\cA^*)( w^{k+1}-\overline w^{k+1})
\right\rangle.
\end{array}
$$
Based on the above equality, one can see from \eqref{ineq:Hx} and \eqref{betamumu} that
\[
\label{deltau}
\|u^{k+1}-\overline u^{k+1}\|_\Omega\le
\mu\|d^k\|.
\]
Since Assumption \ref{assrate} holds, by using \eqref{dkerr}, \eqref{deltau} and the triangle inequality one can get that
$$
\begin{array}{ll}
\|u^{k+1}-\Pi^\Omega_\bK(\overline u^{k+1})\|_\Omega
\le \|u^{k+1}-\overline u^{k+1}\|_\Omega
+\dist_\Omega\big(\overline u^{k+1},\bK\big)
\\[3mm]
\le
\mu\|d^k\|
+\dist_\Omega\big(\overline u^{k+1},\bK\big)
\le\mu\eta_k\|u^k-u^{k+1}\|
+\dist_\Omega(\overline u^{k+1},\bK)
\\[3mm]
\le\mu\eta_k\|u^{k+1}-\Pi^\Omega_\bK(\overline u^{k+1})\|_\Omega
+\mu\eta_k\|u^k-\Pi^\Omega_{\bK}(\overline u^{k+1})\|_\Omega
+\dist_\Omega(\overline u^{k+1},\bK),
\quad\,\forall k\ge \bar k.
\end{array}
$$
Then, by using the fact that
$\|u^{k+1}-\Pi^\Omega_\bK(\overline u^{k+1})\|_\Omega\ge\dist_\Omega(u^{k+1},\bK)$
and \eqref{eep2}, we can obtain that when $k\ge0$,
$$
(1-\mu\eta_k)\dist_\Omega(u^{k+1},\bK)
\; \le \; \mu\eta_k(1+\beta)\dist_\Omega(u^k,\bK) +
\dist_\Omega(\overline u^{k+1},\bK),
$$
which, together with \eqref{ineqmm1}, implies \eqref{rateresult}.
Finally, it is easy to see that $\sup_{k\ge\bar k}\{\vartheta^k\}<1$ from \eqref{dkerr} and \eqref{etarate}.
 This completes the proof.
\qed\end{proof}

\begin{remark}
Note that if $\{\eta_k\}\to 0$ as $k\to\infty$, condition \eqref{etarate} holds eventually for $\bar k$ sufficiently large.
\end{remark}

\subsection{Non-Ergodic Iteration Complexity}
With the inequalities established in the previous subsections, one can easily get the following non-ergodic iteration complexity results for Algorithm \ref{alg:alm}.
\begin{theorem}
\label{thmiter}
Suppose that Assumption \ref{ass:basic} holds. Let $\{u^k=(x^k,w^k)\}$ be the sequence generated by Algorithm \ref{alg:alm} that converges to $u^*:=(x^*,w^*)\in \bK$.
Then, the KKT residual mapping $\cR$ defined in \eqref{residual} satisfies
\[
\label{complexity}
\min_{0\le j\le k}\|\cR(u^{j})\|^2\le \varrho/k
\quad
\mbox{and}
\quad
\lim_{k\to\infty}\big(k\cdot\min_{0\le j\le k}\|\cR(u^{j})\|^2\big)=0,
\]
where the constant $\varrho$ is defined by
\[
\label{varrho}
\begin{array}{ll}
\ds\varrho:=
\max
\left\{
\frac{{12\sigma^2 (\tau-1)^2\|\cA^*\cA\|+3}}
{\tau\sigma^2(2-\tau)},
\frac{2\zeta\|{\what\Sigma_h}+\cS\|
}{\tau\sigma}\right\}{\bf e}
\\[4mm]
\end{array}
\]
with ${\bf e} :=\|u^0_e\|^2_\Omega+2\tau\sigma\|\Theta^{-1/2}\|(\sum_{j=0}^\infty\varepsilon_j)
\big(\|u_e^0\|_\Omega+\mu \sum_{j=0}^\infty\varepsilon_j\big)+4
\sum_{j=1}^{\infty}\varepsilon_j^2\ .$

\end{theorem}
\begin{proof}
From \eqref{result} in Theorem \ref{thm:main}(a) we know that$
\|\overline u_e^{j+1}\|_\Omega\le
\|u_e^{j}\|_\Omega,\,\forall\, j\ge 0$.
Moreover, \eqref{deltau} still holds with $\mu$ being given in \eqref{betamumu}, so that
$$\|u^{j+1}-\overline u^{j+1}\|_\Omega\le
\mu\|d^j\|,\quad\forall\, j\ge 0.$$
Therefore,
$$
\|w_e^{j+1}\|_{\Theta}
\le\|u_e^{j+1}\|_\Omega\le
\|u_e^{j}\|_\Omega+\mu\|d^k\|
\le \|u_e^0\|_\Omega+\mu\sum_{j=0}^\infty\varepsilon_j,
\quad\forall\, j\ge 0.
$$
Consequently, for any $k\ge 0$
\[
\label{sumsum}
\begin{array}{ll}
\ds
\sum_{j=0}^k\langle d^j,w_e^{j+1}\rangle
&
\ds
\le
\|\Theta^{-1/2}\|\left(
\sum_{j=0}^k\|d^j\|\right)\|w_e^{j+1}\|_{\Theta}
\\[4mm]
&
\ds\le
\|\Theta^{-1/2}\|\left(\sum_{j=0}^\infty\varepsilon_j\right)
\left(\|u_e^0\|_\Omega+\mu \sum_{j=0}^\infty\varepsilon_j \right).
\end{array}
\]
Also, from \eqref{result} of Theorem \ref{thm:main}(a) we know that for any $k\ge 0$,
\[
\label{sumsum1}
\begin{array}{ll}
\|u^0_e\|^2_\Omega&
\ds\ge
\|u^0_e\|^2_\Omega
-\|u^{k+1}_e\|^2_\Omega
=
\sum_{j=0}^k
\left(\|u^{j}_e\|^2_\Omega
-\|u^{j+1}_e\|^2_\Omega
\right)
\\[3mm]
&\ds
\ge
\sum_{j=0}^k\left( \|w^{j+1}-w^{j}\|_{\Xi}^{2}
+\frac{2-\tau}{3\tau}\|x^{j+1}-x^{j}\|^2\right)
-2\tau\sigma\sum_{j=0}^k\langle d^j, w_e^{j+1} \rangle.
\end{array}
\]
Moreover, from \eqref{ineqres} we know that
$$
\begin{array}{ll}
\ds
\|\cR(u^{k+1})\|^2
\le
\ds
\frac{ {4\sigma^2 (\tau-1)^2\|\cA^*\cA\|+1}}
{\tau^2\sigma^2}
\| x^{k+1}- x^k\|^2
+\frac{2\zeta\|{\what\Sigma_h}+\cS\|
}{\tau\sigma}
\|w^{k+1}-w^k\|_{\Xi}^2
+4\|d^k\|^2.
\end{array}
$$
Therefore, we can get from \eqref{sumsum} and \eqref{sumsum1} that
$\sum_{j=0}^{\infty}
\|\cR(u^{j+1})\|^2\le
\varrho$.
From here, we can easily get required results in
 \eqref{complexity}.
\qed\end{proof}

\section{The Equivalence Property}
\label{sec:main}

In this section, we establish the equivalence of an inexact block sGS decomposition based multi-block indefinite-proximal ADMM for solving problem \eqref{probmulti} to the inexact indefinite-proximal ALM presented in the previous section.
The iteration scheme of the former has already been briefly sketched in \eqref{iadmm}
in the introduction. Here we shall formally present it as Algorithm \ref{alg:admm}.
\makeatletter
\renewcommand
\thealgorithm{{sGS-iPADMM}}
\makeatother
\begin{algorithm}%[H]
\normalsize
\caption{An inexact block sGS decomposition based indefinite-Proximal ADMM}
\label{alg:admm}
Let $\{\tilde\varepsilon_k\}$ be a summable sequence of nonnegative real numbers, $\tau\in(0,2)$ be the (dual) step-length, and $(x^0,y^0,z^0)\in\dX\times\dom p\times\dY_2\times\cdots\times\dY_s\times\dZ$ be the given initial point.
Choose the self-adjoint linear operators $\cD_i:\dY_i\to\dY_i,\, i=1,\,\ldots,s$.
For $k=0,1,\ldots,$ perform the following steps in each iteration.
\begin{description}
\item[\bf Step 1.] For $i=s,\ldots,2$, compute
$$
y_{i}^{k+\frac{1}{2}}
\approx
\argmin_{y_{i}\in\dY_{i}}
\left\{{\cL}_\sigma
\left(\Big(y^{k}_{< i};y_{i}; y^{k+\frac{1}{2}}_{> i}\Big),z^{k};(x^k,y^k)\right)
+\frac{1}{2}\|y_{i}-y_{i}^k\|_{\cD_{i}}^2 \right\},
$$
such that there exists $\tilde\delta^{k}_i$ satisfying $\|\tilde\delta^{k}_i\| \leq \tilde\varepsilon_k$ and
$$
\tilde\delta^{k}_i\in\partial_{y_{i}}{\cL}_\sigma
\left(\Big(y^{k}_{< i}; y_{i}^{k+\frac{1}{2}}; y^{k+\frac{1}{2}}_{> i}\Big),z^{k};(x^k,y^k)\right)
+\cD_{i}\Big(y_{i}^{k+\frac{1}{2}}-y_{i}^{k}\Big).
$$
\end{description}
%%%%%%%%%%%%%%%%%%%%%%%%%%%%%%%%%%%%%%%%%%%%%%%%%%%%%%%%
\begin{description}
\item[\bf Step 2.] For $i=1,\ldots,s$, compute
$$
y_{i}^{k+1}\approx{\argmin_{y_{i}\in\dY_{i}}}
\left\{{\cL}_\sigma
\left(\Big(y^{k+1}_{< i};y_{i}; y^{k+\frac{1}{2}}_{>i}\Big),z^{k};(x^k,y^k)\right)
+\frac{1}{2}\|y_{i}-y_{i}^k\|_{\cD_{i}}^2 \right\},
$$
such that there exists $\delta^k_i$ satisfying $\|\delta^k_i\| \leq \tilde\varepsilon_k$ and
$$
\delta^k_i\in\partial_{y_{i}}{\cL}_\sigma
\left(\Big(y^{k+1}_{<i}; y_{i}^{k+1}; y^{k+\frac{1}{2}}_{> i}\Big),z^{k};(x^k,y^k)\right)
+\cD_{i}\Big(y_{i}^{k+1}-y_{i}^{k}\Big).
$$
\end{description}
%%%%%%%%%%%%%%%%%%%%%%%%%%%%%%%%%%%%%%%%%%%%%%%%%%%%%%%%
\begin{description}
\item[\bf Step 3.]
Compute $z^{k+1}\approx\argmin\limits_{z\in\dZ}
\left\{\cL_{\sigma}\left(y^{k+1},z;(x^k,y^k)\right)\right\}$,
such that $\|\gamma^k\|\le\tilde\varepsilon_k$ with
\[
\label{gammak}
\gamma^k:=\nabla_z\cL_{\sigma}\left(y^{k+1} ,z^{k+1};(x^k,y^k)\right)
=\cG x^k - b + \sig \cG (\cF^* y^{k+1} + \cG^* z^{k+1}-c).
\]
\end{description}
\begin{description}

\item[\bf Step 4.]
Compute $x^{k+1}:=x^k+\tau\sigma\left(\cF^*y^{k+1}+\cG^* z^{k+1}-c\right)$.
\end{description}
\end{algorithm}

Recall that the KKT system of problem \eqref{probmulti} is defined by
\[
\label{kkt}
0\in
\begin{pmatrix}
\partial p(y_1)\\
0
\end{pmatrix}+\nabla f(y)+\cF x
,\quad
\cG x-b=0,\quad
\cF^*y+\cG^*z=c.
\]
We make the following assumption on problem \eqref{probmulti} throughout this section.
\begin{assumption}
\label{ass}
The solution set to the KKT system \eqref{kkt} is nonempty.
\end{assumption}
Note that if Slater's constraint qualification (SCQ) holds for problem \eqref{probmulti}, i.e.,
$$
\big\{(y,z)\ |\ y_1\in\ri\,(\dom p),\ \cF^{*}y+\cG^{*}z=c\big\}\ne \emptyset,
$$
then we know from \cite[Corollaries 28.2.2 \& 28.3.1]{rocbook} that a vector $(y,z)\in\dY\times\dZ$ is a solution to problem \eqref{probmulti} if and only if there exists a Lagrangian multiplier $ x\in\dX$ such that
$(x,y,z)$ is a solution to the KKT system \eqref{kkt}.
Therefore, {Assumption \ref{ass}} holds if the SCQ holds and \eqref{probmulti} has at least one optimal solution. Moreover, for any $(x,y,z)\in\dX\times\dY\times\dZ$ satisfying \eqref{kkt}, we know
from \cite[Corollary 30.5.1]{rocbook}
that $(y, z)$ is an optimal solution to problem \eqref{probmulti} and $x$ is an optimal solution to its dual problem.

Recall that the majorized augmented Lagrangian function of problem \eqref{probmulti} was given in \eqref{lagr}.
Note that one can always write $\cF x= (\cF_1x;\ldots;\cF_s x)$, $\forall x\in\dX$ with each $\cF_i:\dX\to\dY_i$ being a given linear mapping.
For later discussions, we symbolically decompose the self-adjoint linear operator $\what\Sigma^f$ in the following from
\[
\label{sigfdec}
\what\Sigma^f=
\begin{pmatrix}
\ \what\Sigma^f_{11}\ &
\ \what\Sigma^f_{12}\ &
\ \cdots \ &
\ \what\Sigma^f_{1s}\ \\[3mm]
\what\Sigma^f_{21}&\what\Sigma^f_{22}&\cdots &\what\Sigma^f_{2s}\\[3mm]
\vdots&\vdots&\ddots&\vdots \\[3mm]
\what\Sigma^f_{s1}&\what\Sigma^f_{s2}&\cdots &\what\Sigma^f_{ss}
\end{pmatrix}
\quad\mbox{with}
\quad
\what\Sigma^f_{ij}:\dY_j\to\dY_i,\
\forall 1\le i,j\le s.
\]
Based on the above decomposition, we make the following assumption on choosing the proximal terms in Algorithm \ref{alg:admm} .
\begin{assumption}
\label{assp}
The self-adjoint linear operators
$\cD_i:\dY_i\to\dY_i,i=1,\,\ldots,s$ in Algorithm \ref{alg:admm} are chosen such that
\[
\label{indcondition}
\frac{1}{2}\what\Sigma^f_{ii}+\sigma\cF_i\cF_i^*+\cD_i\succ 0
\quad\mbox{and}\quad
\cD:=\diag(\cD_1,\ldots,\cD_s)\succeq-\frac{1}{2}\what\Sigma^f.
\]
\end{assumption}

We are now ready to prove the equivalence of Algorithm 1 and Algorithm \ref{alg:admm}
for solving problem \eqref{probmulti}.
We begin by applying the inexact block sGS decomposition technique in \cite[Theorem 1]{lisgs}
 to express the procedure for computing $y^{k+1}$ in Steps 1 and 2 of Algorithm \ref{alg:admm} in a more compact fashion.
For this purpose we define the following linear operator
\[
\label{defno}
\cN:=\what\Sigma^f+\sigma\cF\cF^*+\cD.
\]
Note that the self-adjoint linear operator $\cN$ is positive semidefinite, if Assumption \ref{assp} holds.
Moreover, as can be seen from \eqref{lagr}, for any given $(x,y',z)\in\dX\times\dY\times\dZ$, the linear operator $\cN$ contains all the quadratic information of
$$
\cL_\sigma\left(y,z;(x,y')\right)+\frac{1}{2}\|y-y'\|_\cD^2
$$
with respect to $y$.
Based on \eqref{sigfdec}, the linear operator $\cN$ can be decomposed as $\cN=\cN_d+\cN_u+\cN_u^*$ with $\cN_d$ and $\cN_u$ being the block-diagonal and the strict block-upper triangular parts of $\cN$, respectively, i.e.,
$$
\cN_d:=\diag(\cN_{11},\ldots,\cN_{ss})
\quad\mbox{with}
\quad \cN_{ii}:=\what\Sigma^f_{ii}+\sigma\cF_i\cF_i^*+\cD_i,\quad
i=1,\ldots,s
$$
and
\[
\label{Hu}
\cN_u:=
\begin{pmatrix}
\ 0\ &\ \cN_{12}\
&\cdots&\ \cN_{1s}
\\[3mm]
\ 0&0&\ddots&\vdots\\[3mm]
\ \vdots&\vdots&\ddots&\cN_{(s-1)s}
\\[3mm]
\ 0&\ 0&\ \cdots&0
\end{pmatrix}
\quad
\mbox{with}
\quad
\cN_{ij}=\what\Sigma^f_{ij}+\sigma\cF_i\cF_j^*,\quad \forall\; 1\leq i < j\leq s.
\]
For convenience, we denote in Algorithm \ref{alg:admm} for each $k\ge 0$,
$\tilde\delta_1^k:=\delta_1^k$,
$\tilde\delta^k:=(\tilde\delta^k_1,\tilde\delta_k^2\ldots,\tilde\delta_s^k)$
and $\delta^k:=(\delta^k_1,\ldots,\delta^k_s)$.
Suppose that Assumption \ref{assp} holds. We can define the sequence $\{\delta_\sgs^k\}\in\dY$ by
\[
\label{delta}
\delta_\sgs^k:=\delta^{k}+\cN_u\cN_d^{-1}(\delta^k-\tilde\delta^k).
\]
Moreover, we can define the linear operator
\[
\label{defhm}
\cN_\sgs:=\cN_u\cN_d^{-1}\cN_u^*.
\]
Based on the above definitions, we have the following result, which is a direct consequence of \cite[Theorem 1]{lisgs}.
\begin{lemma}
\label{prop:y}
Suppose that Assumption \ref{assp} holds.
The iterate $y^{k+1}$ in Step 2 of Algorithm \ref{alg:admm} is the unique solution to the perturbed proximal minimization problem given by
\[
\label{ykp1}
y^{k+1}=\argmin_{y\in\dY}\left\{
\cL_{\sigma}\Big(y,z^k;\big(x^k,y^k\big)\Big)+\frac{1}{2}\|y-y^k\|^2_{\cD+\cN_\sgs}
-\langle \delta_\sgs^k,y\rangle\right\}.
\]
Moreover, it holds that
$\cN+\cN_\sgs= (\cN_d+\cN_u)\cN_d^{-1}(\cN_d+\cN_u^*)\succ 0$.
\end{lemma}

\begin{remark}
From \eqref{ykp1} one can get the interpretation of the linear operator $\cN_\sgs$ defined in \eqref{defhm}. That is, by adding the proximal term $\frac{1}{2}\|y-y^k\|^2_\cD$ to the majorized augmented Lagrangian function and conduct one cycle of the block sGS-type block coordinate minimization via Steps 1 and 2 in Algorithm \ref{alg:admm}, the resulted $y^{k+1}$ is then an inexact solution to the following problem
$$
\min_{y\in\dY}\left\{
\cL_{\sigma}\Big(y,z^k;\big(x^k,y^k\big)\Big)+\frac{1}{2}\|y-y^k\|^2_{\cD}+\frac{1}{2}\|y-y^k\|^2_{\cN_\sgs}
\right\},
$$
where the proximal term $\frac{1}{2}\|y-y^k\|^2_{\cN_\sgs}$ is generated due to the sGS-type iteration with the linear operator $\cN_\sgs$ being defined by \eqref{defhm} and \eqref{defno},
while $\delta_{\sgs}^k$ defined in \eqref{delta} represents the error accumulated from $\tilde\delta^k$ and $\delta^k$ after one cycle of the sGS-type update.
\end{remark}

The following elementary result\footnote{This lemma can be directly verified via the singular value decomposition of the linear operator $\cG$ and some basic calculations from linear functional analysis.} will be frequently used later.
\begin{lemma}
\label{psdsub}
The self-adjoint linear operator $\cG\cG^*$ is nonsingular (positive definite) on the subspace $\range(\cG)$ of $\dZ$.
\end{lemma}

Now, we start to establish the equivalence between Algorithm \ref{alg:admm} and Algorithm \ref{alg:alm}.
The first step is to show that the procedure of obtaining $(y^{k+1},z^{k+1})$ in Algorithm \ref{alg:admm} can be viewed as the procedure of getting $w^{k+1}$ in Algorithm \ref{alg:alm}. For this purpose, we define the block diagonal linear operator $\cT:\dY\times\dZ\to\dY\times\dZ$ by
\[
\label{tyz}
\cT\big(y;z\big):=
\begin{pmatrix}
\left(\cD+\cN_{\sgs}+\sigma\cF\cG^* [\cG\cG^*]^{\dag}\cG\cF^*\right)y \\[3mm] 0
\end{pmatrix},\quad\forall\,\big(y,z\big)\in\dY\times\dZ.
\]
Moreover, we define the sequence $\left\{\Delta^k\right\}$ in $\dY$ by
\[
\label{hdelta}
\Delta^k:=\delta_{\sgs}^{k}-\cF\cG^*
 [\cG\cG^*]^{\dag}
\left(\gamma^{k-1} -\gamma^{k} -\cG (x^{k-1}-x^{k})\right), \quad k\ge 0
\]
with the convention that
\[
\label{xgm1}
\left\{
\begin{array}{l}
x^{-1}:=x^{0}-\tau\sigma(\cF^*y^{0}+\cG^*z^{0}-c),
\\[3mm]
\gamma^{-1}:= -b+\cG x^{-1}+\sigma \cG(\cF^*y^{0}+\cG^*z^{0}-c).
\end{array}
\right.
\]
Based on the above definitions and Lemma \ref{prop:y}, we have the following result.
\begin{proposition}
Suppose that Assumption \ref{assp} holds.
Then,
\begin{itemize}
\item[\bf (a)] Algorithm \ref{alg:admm} is well-defined;
\item[\bf (b)] the sequence $\{(x^k,y^k,z^k)\}$ generated by Algorithm \ref{alg:admm} satisfies
\[
\label{subg}
\left(\Delta^{k}; \gamma^{k}\right)
\in\partial_{(y,z)}\cL_{\sigma}\left(\big(y^{k+1},z^{k+1}\big);\big(x^k,y^k\big)\right)+
\cT\left(
y^{k+1}-y^k; z^{k+1}-z^k\right),
\quad\forall\, k\ge 0.
\]
\end{itemize}
 \end{proposition}
\begin{proof}
{\bf(a)} Since Assumption \ref{assp} holds, it is easy to see from Lemma \ref{prop:y} that Steps $1$ and $2$ in algorithm \ref{alg:admm} are well-defined for any $k\ge 0$.
Moreover, from \eqref{gammak} we know that Step $3$ of Algorithm \ref{alg:admm} is well-defined if, for any $k\ge 0$, the following linear system with respect to $z$
\[
\label{leaux}
\cG x^k - b + \sig \cG (\cF^* y^{k+1} + \cG^* z-c)=0
\]
has a solution.
Since $b\in \range(\cG)$, we know that $(b - \cG x^k)/\sigma - \cG(\cF^* y^{k+1} - c) \in \range(\cG)$. Therefore, Lemma \ref{psdsub} implies that the linear system
$$
\cG\cG^* z = (b - \cG x^k)/\sigma - \cG(\cF^* y^{k+1} - c)
$$
or equivalently the linear system \eqref{leaux}, has a solution.
Consequently, Algorithm \ref{alg:admm} is well-defined.

\smallskip
\noindent{\bf (b)} From \eqref{gammak} and \eqref{xgm1} we know that for any $k\ge 0$,
\[
\label{subopt2}
\gamma^{k-1}
= -b+\cG x^{k-1}+\sigma \cG(\cF^*y^{k}+\cG^*z^{k}-c)
\]
so that $\gamma^{k-1}\in\range(\cG)$ and
$\cG\cG^*z^{k}
= (\gamma^{k-1} +b-\cG x^{k-1})/\sigma- \cG\cF^*y^{k}+\cG c$.
Hence,
$$
 \cG\cG^* (z^k-z^{k+1})=
\frac{1}{\sigma}
(\gamma^{k-1}-\gamma^{k} -\cG x^{k-1}+\cG x^{k})- \cG\cF^*(y^k-y^{k+1}), \quad
\forall\, k\ge 0.
$$
Therefore, one can get\footnote{This can be routinely derived by using the singular value decomposition of $\cG$ and the definition of the Moore-–Penrose pseudoinverse.} that for any $k\ge 0$,
\[
\label{pterm}
\begin{array}{l}
\sigma\cF\cG^*(z^{k}-z^{k+1})
\\[3mm]
\ds
=\cF\cG^*
 [\cG\cG^*]^{\dag}
(\gamma^{k-1}-\gamma^{k} -\cG (x^{k-1}-x^{k}))
+\sigma\cF\cG^* [\cG\cG^*]^{\dag}\cG\cF^*(y^{k+1}-y^k).
\end{array}
\]
From \eqref{ykp1} in Lemma \ref{prop:y} we know that, for any $k\ge 0$,
\[
\label{pf:1}
\begin{array}{ll}
\delta_{\sgs}^k&\in\partial_y\cL_\sigma\left(y^{k+1},z^k;(x^k,y^k)\right)
+\big(\cD+\cN_\sgs\big)(y^{k+1}-y^k)
\\[3mm]
&=\partial_y \cL_\sigma\left(y^{k+1},z^{k+1};\big(x^k,y^k\big)\right)+\big(\cD+\cN_\sgs\big)(y^{k+1}-y^k)
+\sigma\cF\cG^*\big(z^k-z^{k+1}\big).
\end{array}
\]
Then, by substituting \eqref{pterm} into \eqref{pf:1} and using the definition of $\Delta^k$ in \eqref{hdelta}, one has that
$$
\Delta^k \in\partial_y \cL_\sigma\left(y^{k+1},z^{k+1};(x^k,y^k)\right)+\left(\cD+\cN_\sgs+\sigma\cF\cG^* [\cG\cG^*]^{\dag}\cG\cF^*\right)(y^{k+1}-y^k),
$$
which, together with \eqref{gammak}, implies that \eqref{subg} holds. This completes the proof.
\qed\end{proof}

The following important result will be used later.

\begin{proposition}
\label{summable}
Suppose that Assumptions \ref{ass} and \ref{assp} hold. Let $\{(x^k,y^k,z^k)\}$ be the sequence generated by Algorithm \ref{alg:admm}.
Define $\xi_0 := \norm{b - \cG x^0}$ and
$$\begin{array}{l}
\xi_k := |1 - \tau|^k \xi_0 + \tau \sum_{i=1}^k |1 - \tau|^{k-i} \tilde{\varepsilon}_{i-1},
\quad\forall\, k\ge1.
\end{array}
$$
Then, it holds that for all $k\ge 0$, $\norm{b - \cG x^k} \le \xi_k$ and
$$\sum_{k=0}^{\infty} \norm{b - \cG x^k} \le \sum_{k=0}^{\infty} \xi_k < +\infty.
$$
\end{proposition}
\begin{proof}
We know from Step 4 of Algorithm \ref{alg:admm} and \eqref{xgm1} that
$$
x^{k}= x^{k-1}+\tau\sigma\big(\cF^*y^{k}+\cG^*z^{k}-c\big),
\quad\forall\,
k\ge 0.
$$
Hence, one has that
$$
b-\cG x^{k}
=b-\cG x^{k-1}-\tau\sigma \cG\big(\cF^*y^{k}+\cG^*z^{k}-c\big), \quad \forall\, k\ge 0.
$$
Moreover, from \eqref{xgm1} and \eqref{subopt2} we know that
$$
\tau(\gamma^{k-1}+b-\cG x^{k-1})=
\tau\sigma \cG\big(\cF^*y^{k}+\cG^*z^{k}-c\big), \quad \forall k\ge 0.
$$
Thus, by combining the above two equalities together, one can get
$$
b-\cG x^{k}
=b-\cG x^{k-1}-\tau\big(\gamma^{k-1}+b-\cG x^{k-1}\big)
=(1-\tau)\big(b-\cG x^{k-1}\big)-\tau \gamma^{k-1},
 \quad \forall\, k\ge 0.
$$
Consequently, it holds that
$$
\|b-\cG x^{k}\|
\le
|1-\tau|\, \|b-\cG x^{k-1}\|+\tau\|\gamma^{k-1}\|,
\quad\forall\, k\ge 0,
$$
and hence
\[
\label{bmgs}
\|b-\cG x^{k}\|\le|1-\tau|^{k}\,\|b-\cG x^{0}\|+\tau\sum_{i=1}^{k}|1-\tau|^{k-i}\,\|\gamma^{i-1}\|
\le\xi_k,
\quad\forall k\ge 0.
\]
Note that $\tau\in(0,2)$. It is easy to see that
$$
\begin{array}{ll}
\ds
\sum_{k=0}^{\infty}\|b-\cG x^{k}\|
&\ds
\le \sum_{k=0}^{\infty} \xi_k \le
\left(\sum_{k=0}^{\infty}|1-\tau|^{k}\right)\xi_0+\tau\sum_{k=1}^{\infty}\sum_{i=1}^{k}|1-\tau|^{k-i}\tilde\varepsilon_{i-1}
\\[4mm]
&\ds
\le
\left(\sum_{k=0}^{\infty}|1-\tau|^{k}\right)\xi_0
+\tau\left(\sum_{ {k=0}}^{\infty}|1-\tau|^{k}\right)
\left(\sum_{i=0}^{\infty}
\tilde\varepsilon_{i}
\right)
<+\infty,
\end{array}
$$
which completes the proof.
\qed\end{proof}

Now, we start to show that the sequence $\{(x^k,y^k,z^k)\}$ generated by Algorithm \ref{alg:admm} can be viewed as a sequence generated by Algorithm \ref{alg:alm} from the same initial point. For this purpose, we define the space $\dV:= \dY\times\range(\cG)$, and we define the linear operators $\cB:\dX\to\dV$ and $\cP:\dV\to\dV$ by
\[
\label{hpo}
\cB x:=\big(\cF x;\cG x\big),\ \forall\, x\in\dX
\quad\mbox{ and }\quad
\cP(y,z):=\left(\,\what\Sigma^f y\,;\, 0\right),\ \forall\, (y,z)\in\dV.
\]
Moreover, we define the closed proper convex function $\phi:\dV\to(-\infty,+\infty]$ by
$$
\phi(v)
=\phi(y,z):=p(y_1)+f(y)-\langle b,z\rangle,
\quad
\forall\, v=(y,z)\in\dV
$$
and define
\[
\label{newcl}
\cL_\sigma\left(v;(x,v')\right)
:=\cL_\sigma\left(y,z;(x,y')\right),
\quad\forall\,
v=(y,z)\in\dV,\ v'=(y',z')\in\dV.
\]
Based on the above definitions, problem \eqref{probmulti} can be viewed as an instance of problem \eqref{prob1}.
In this case, the following result is for the purpose of viewing Algorithm \ref{alg:admm} as an instance of Algorithm \ref{alg:alm}.

\begin{theorem}
\label{theomain}
Suppose that Assumptions \ref{ass} and \ref{assp} hold.
Let $\{(x^k,y^k,z^k)\}$ be the sequence generated by Algorithm \ref{alg:admm}.
Define
\[
\label{defvk}
v^k:=\left(y^k;\Pi_{\range(\cG)}(z^k)\right),\quad\forall\, k\ge 0.
\]
Then, for any $k\ge 0$, it holds that
\begin{itemize}
\item[\bf(a)] the linear operators $\cT$, $\cB$ and $\cP$ defined in \eqref{tyz} and \eqref{hpo} satisfy
\[
\label{proxixx}
\begin{array}{l}
\cT\succeq -\frac{1}{2}{\cP}
\quad\mbox{and}\quad
\left\langle v,
\left(\frac{1}{2}{\cP}+\sigma\cB\cB^*+\cT\right)v
\right\rangle>0,\quad \forall\, v\in\dV\setminus\{0\};
\end{array}
\]
\item[\bf(b)]
there exists a sequence of nonnegative real numbers $\{\what\varepsilon_k\}$, such that
$$
\|(\Delta^{k};\gamma^{k})\|\le\what\varepsilon_k
\quad\mbox{and}
\quad
\sum_{k=0}^{\infty}\what\varepsilon_k<+\infty;
$$
\item[\bf(c)] it holds that
$$
v^{k+1}
\approx\argmin_{v\in\dV}
\left\{\cL_\sigma
\left(v;(x^k,v^k)\right)
+\frac{1}{2}\|v-v^k\|_{\cT}^2 \right\}
$$
in the sense that
$$
\left(\Delta^{k};\gamma^{k}\right)
\in\partial_{v} \cL_{\sigma}
\left(v^{k+1};\big(x^k,v^k\big)\right)+
\cT\left(v^{k+1}-v^k\right)
\quad\mbox{and}\quad
\left\|(\Delta^{k};\gamma^{k})\right\|
\le\what\varepsilon_k.$$
\end{itemize}
\end{theorem}

\begin{proof}
{\bf(a)}
According to \eqref{indcondition} in Assumption \ref{assp}
we know that $\cD \succeq -\frac{1}{2}\what\Sigma^f$.
Moreover, from \eqref{defhm} we know that $\cN_{\sgs}\succeq 0$. Thus, one can readily see from \eqref{tyz} and \eqref{hpo} that $\cT\succeq-\frac{1}{2}\cP$.
On the other hand, one can symbolically do the decomposition that
$$
\frac{1}{2}{\cP}+\sigma\cB\cB^*+\cT
=\begin{pmatrix}
\frac{1}{2}{\what\Sigma^f}
+\sigma\cF\cF^*+\cD +\cN_{\sgs}
+\sigma\cF\cG^*[\cG\cG^*]^{\dag}\cG\cF^*\quad
& \sigma\cF\cG^* \\[3mm]
\sigma\cG\cF^* & \sigma\cG\cG^*
\end{pmatrix}.
$$
From Lemma \ref{psdsub}, we know that $\cG\cG^*$ is nonsingular on the $\range(\cG)$.
Therefore, by using the definition of $\dV$ and the Schur complement condition for ensuring the positive definiteness of a linear operator, we only need to show that
$\frac{1}{2}{\what\Sigma^f}+\sigma\cF\cF^*
+\cD+\cN_{\sgs}\succ 0$ on $\dY$.
Suppose on the contrary that it is not positive definite. Then, there exists a nonzero vector $y\in\dY$ such that
$$
\begin{array}{l}
\left\langle y, \left(\frac{1}{2}{\what\Sigma^f}+\sigma\cF\cF^*
+\cD+\cN_{\sgs} \right)y\right\rangle
=\left\langle y, \left(\frac{1}{2}{\what\Sigma^f}+\cD +\sigma\cF\cF^*\right)y\right\rangle
+\left\langle y, \cN_{\sgs} y\right\rangle
=0.
\end{array}
$$
From \eqref{indcondition} of Assumption \ref{assp} and \eqref{defhm} we know that $\frac{1}{2}\what\Sigma^f + \cD + \sig\cF\cF^*\succeq 0$ and $\cN_\sgs \succeq 0$, so that
$$\begin{array}{l}
\left\langle y, \left(\frac{1}{2}{\what\Sigma^f}+\cD +\sigma\cF\cF^*\right)y\right\rangle=0=
\left\langle y, \cN_{\sgs} y\right\rangle.
\end{array}
$$
Then, by using \eqref{defhm} we can get that $\cN_u^*y=0$.
This, together with \eqref{Hu}, implies that
\[
\label{eq:split}
\begin{array}{ll}
0&=\left\langle y, \left(\frac{1}{2}{\what\Sigma^f}+\cD +\sigma\cF\cF^*\right)y\right\rangle
\\[3mm]
&
=\frac{1}{2}\left\langle y,
\left(\what\Sigma^f+\sigma\cF\cF^*\right)y\right\rangle
+\left\langle y, \left(\frac{1}{2}\sigma\cF\cF^*
+\cD\right)y\right\rangle
\\[3mm]
&
=\frac{1}{2}
\left\langle y, \left(\what\Sigma^f+\sigma\cF\cF^*\right)_d y\right \rangle
+\left\langle y, \left(\frac{1}{2}\sigma\cF\cF^*
+\cD\right)y\right\rangle
\\[3mm]
&
=\left\langle y, \left(\frac{1}{2}(\what\Sigma^f)_d\,+\cD\right) y\right\rangle
+\frac{\sigma}{2}\left\langle y, (\cF\cF^*)_d y\right\rangle
+\frac{\sigma}{2}\left\langle y, \cF\cF^* y\right\rangle,
\end{array}
\]
where
$$
\begin{array}{rcl}
(\what\Sigma^f+\sigma\cF\cF^*)_d &{:=} &
\diag\big(\what\Sigma^f_{11}+
\sigma\cF_1\cF^*_1,\ldots,\what\Sigma^f_{ss}+
\sigma\cF_s\cF^*_s\big),
\\[3mm]
(\cF\cF^*)_d& {:=}&
\diag\big(\cF_1\cF^*_1,\ldots,\cF_s\cF^*_s\big).
\end{array}
$$
Since $\cD\succeq -\frac{1}{2}\what\Sigma^f$ implies $\frac{1}{2}(\what\Sigma^f)_d\,+\cD\succeq 0$, we obtain from \eqref{eq:split} that
$$
\begin{array}{l}
\left\langle y, \left(\frac{1}{2}(\what\Sigma^f)_d\,+\cD\right) y\right\rangle
=\frac{\sigma}{2}\left\langle y, (\cF\cF^*)_d\, y\right\rangle
=\frac{\sigma}{2}\left\langle y, \cF\cF^* y\right\rangle=0,
\end{array}
$$
which contradicts the requirement in Assumption \ref{assp} that $\frac{1}{2}\what\Sigma^f_{ii}+\sigma\cF_i\cF_i^*+\cD_i\succ 0$ for all $i=1,\ldots,s$.
Therefore, it holds that $\frac{1}{2}{\what\Sigma^f}+\sigma\cF\cF^*
+\cD+\cN_{\sgs}\succ 0$, and this completes the proof of (a).

\smallskip
\noindent{\bf (b)} From the definition of $\{\Delta^k\}$ in \eqref{hdelta} one can see that for all $k\ge 0$,
$$
\|\Delta^k\|
\le\|\delta_{\sgs}^{k}\|+\|\cF\cG^*
 [\cG\cG^*]^{\dag}\|
\|\gamma^{k-1}-\gamma^{k} -\cG (x^{k-1}-x^{k})\|.
$$
Then, by using the fact that $\max\{\|\tilde\delta_i^k\|, \norm{\delta_i^k}, \norm{\gamma^k}\}\le\tilde\varepsilon_k$,
we can get from Proposition \ref{summable}
and the definition of $\delta_{\sgs}^{k}$ in \eqref{delta} that for all $k\ge 1$,
$$
\begin{array}{ll}
\|(\Delta^k;\gamma^k)\| \le
\|\gamma^k\|+\|\Delta^k\|
\\[3mm]\le \what\varepsilon_k: = (s+1)\tilde\varepsilon_k
+2s\|\cN_u\cN_d^{-1}\|\tilde\varepsilon_k
+\|\cF\cG^*
[\cG\cG^*]^{\dag}\|
\big(\tilde\varepsilon_{k-1}+ \xi_{k-1}
+\tilde\varepsilon_k + \xi_k\big).
\end{array}
$$
Moreover, we define $\what\varepsilon_0:=\|(\Delta^0;\gamma^0)\|$.
Then, according to Proposition \ref{summable} and the fact that the sequence $\{\tilde\varepsilon_k\}$ is summable, we know that $\sum_{k=0}^{\infty}\what\varepsilon_k<+\infty$.

{\bf(c)}
According to \eqref{tyz}, \eqref{subg} and \eqref{newcl}, we only need to show that
$$
\partial_{(y,z)}\cL_{\sigma}\left(\big(y^{k+1},z^{k+1}\big);\big(x^k,y^k\big)\right)
=
\partial_{(y,z)}\cL_{\sigma}\left(\Big(y^{k+1},\Pi_{\range(\cG)}(z^{k+1})\Big);\big(x^k,y^k\big)\right),
\quad
\forall\, k\ge 0.
$$
From \eqref{mfy} and \eqref{lagr} we can get that
$$
\partial_{y}\cL_\sigma(y,z;(x,y'))
=
\begin{pmatrix}
\partial_{y_1}p(y_1)
\\
0
\end{pmatrix}
+\nabla f(y')+\what\Sigma_f(y-y')
+\cF x
+{\sigma}\cF(\cF^*y+\cG^*z-c)
$$
and
$$
\nabla_{z}\cL_\sigma(y,z;(x,y'))
=-b+\cG x+{\sigma}\cG(\cF^*y+\cG^*z-c).
$$
Therefore, by using the fact that $\cG^*z^{k+1}=\cG^*\Pi_{\range(\cG)}(z^{k+1})$, $\forall\,k\ge 0$, we know that part (c) of the theorem holds. This completes the proof.
\qed
\end{proof}

\begin{remark}
One can see that in Algorithm \ref{alg:admm}, the sequence $\{(x^k,y^k,z^k)\}$ was generated, while the sequence
$\{\Pi_{\range(\cG)}(z^k)\}$ has never been explicitly calculated.
Note that once $z^k$ is computed, only the vector $\cG^*z^k$ is needed during the next iteration, instead of $z^k$ itself. Since
$\cG^*z^k=\cG^*\Pi_{\range(\cG)}(z^{k}), \forall\, k\ge 0 $, one may view the sequence $\{\Pi_{\range(\cG)}(z^{k})\} \in \range(\cG)$ as a shadow sequence of $\{z^k\}$. It has never been explicitly computed, but still plays an important role on establishing the convergence of the algorithm. In fact, similar observations have been made and extensively used in \cite{lixd2014,liqsdpnal}.
\end{remark}

By combining the results of Theorem \ref{thm:main}
and Theorem \ref{theomain}, one can readily get the following convergence theorem of Algorithm \ref{alg:admm}.

\begin{theorem}
\label{maintheo}
Suppose that Assumptions \ref{ass} and \ref{assp} hold. Let $\{(x^k,y^k,z^k)\}$ be the sequence generated by Algorithm \ref{alg:admm}.
Then,
\begin{itemize}
\item[\bf(a)]
the sequence
$\left\{\left(y^k,\Pi_{\range(\cG)}(z^k)\right)\right\}$
converges to a solution to problem \eqref{probmulti}
and the sequence $\{x^k\}$ converges to a solution to the dual of \eqref{probmulti};

\item[\bf(b)]
any accumulation point of the sequence $\{(y^k, z^k)\}$ is a solution to problem \eqref{probmulti};

\item[\bf(c)] the sequence $\{
p(y_1^k)+f(y^k)-\langle b,z^k\rangle\}$ of the objective values converges to the optimal value of problem \eqref{probmulti}, and
$$
\lim_{k\to\infty} (\cF^*y^k+\cG^*z^k-c) = 0;
$$

\item[\bf(d)] it holds with $\bK$ being the solution set to the KKT system \eqref{kkt} that
$$
\lim_{k\to\infty}
\dist\left((x^k,y^k,z^k), \bK\right)=0;
$$

\item[\bf(e)] if the linear operator $\cG$ is surjective,
the whole sequence $\{(x^k,y^k,z^k)\}$ converges to a solution to the KKT system \eqref{kkt} of problem \eqref{probmulti}.
\end{itemize}
\end{theorem}

\begin{proof}
{\bf(a)} Note that the sequence $\left\{v^k = \left(y^k;\Pi_{\range(\cG)}(z^k)\right) \right\}$ defined in \eqref{defvk} lies in $\dY\times\range(\cG)$.
By using Theorem \ref{theomain}(c), one can treat the sequence $\{(x^k,v^k)\}$ generated by Algorithm \ref{alg:admm}
as the one generated by Algorithm \ref{alg:alm} with the given initial point
$(x^0,v^0)$.
In addition, \eqref{proxixx} in Theorem \ref{theomain} guarantees that condition \eqref{proxi} in Assumption
\ref{ass:basic} holds. Thus, by Theorem \ref{thm:main}, the sequence $\{(x^k,v^k)\}$ converges to a solution to the KKT system \eqref{kkt}, i.e., the sequences $\left\{\left(y^k,\Pi_{\range(\cG)}(z^k)\right)\right\}$
and $\{x^k\}$ converge to a solution to problem \eqref{probmulti} and its dual, respectively.

\smallskip
\noindent{\bf(b)} From {(a)}, we see that $\lim_{k\to\infty}(x^k,y^k,\Pi_{\range(\cG)}(z^k)) = (x^*,y^*,z^*)$ which is a solution to the KKT system \eqref{kkt}. Since $\cG^*z^k=\cG^*\Pi_{\range(\cG)}(z^{k}), \forall\, k\ge 1 $, any accumulation point, say $z^{\infty}$ of $\{z^k\}$ satisfies $\cG^*z^\infty=\cG^*z^*$. Then, it is easy to verify that $(x^*,y^*,z^\infty)$ also satisfy the KKT system \eqref{kkt}.
Therefore, $(y^*,z^\infty)$ is a solution to problem \eqref{probmulti}.

\smallskip
\noindent{\bf(c)}
From {(a)} and the fact that the objective function of problem \eqref{probmulti} is continuous on its domain, we know that $\{p(y_1^k)+f(y^k)-\langle b,\Pi_{\range(\cG)}(z^k)\rangle\}$ converges to the optimal value of problem \eqref{probmulti}.
Since $b\in \range(\cG)$, it holds that for any $k\ge 1$, $\inprod{b}{z^k} = \inprod{b}{\Pi_{\range(\cG)}(z^k)}$. Thus,
$$
p(y_1^k)+f(y^k)-\langle b,\Pi_{\range(\cG)}(z^k) \rangle =
p(y_1^k)+f(y^k)-\langle b, z^k \rangle, \quad \forall k\ge 1.
$$
Therefore, the sequence $\{p(y_1^k)+f(y^k)-\langle b,z^k\rangle\}$ converges to the optimal value of problem \eqref{probmulti}.
Meanwhile, since $\cG^*z^k=\cG^*\Pi_{\range(\cG)}(z^{k})$, we further have that
$$
\lim_{k\to \infty} \left(\cF^*y^k+\cG^*z^k-c\right) =
\lim_{k\to\infty} \left(\cF^*y^k+\cG^*\Pi_{\range(\cG)}(z^{k})-c\right) = 0.
$$

\smallskip
\noindent{\bf(d)}
From {(a)}, we have that $(x^*,y^*,z^*)$, the limit point of $\{(x^k,y^k,\Pi_{\range(\cG)}(z^k))\}$, is a solution to the KKT system \eqref{kkt}, i.e., $(x^*,y^*,z^*)\in {\bf K}$. Since $\cG^*\big(z^k-\Pi_{\range(\cG)}(z^k)\big) =0$ for any $k\ge 1$, it is not difficult to see that
$$
\left(x^*,y^*,z^*+\big(z^k - \Pi_{\range(\cG)}(z^k)\big)\right) \in {\bf K}, \quad \forall k\ge 1.
$$
Therefore, it holds for all $k\ge 1$
$$
\dist\left((x^k,y^k,z^k),
\bK\right)
\le
\|x^k-x^*\|+\|y^k-y^*\|+\norm{\Pi_{\range(\cG)}(z^k)-z^*}
$$
and $\lim_{k\to \infty} \dist\left((x^k,y^k,z^k),
\bK\right) = 0$.

\smallskip
\noindent{\bf(e)} In this case, it holds that $\range(\cG)=\dZ$ and $z^{k}=\Pi_{\range(\cG)}(z^{k}),\,\forall\, k\ge 0$. The result follows from (a), which completes the proof of the theorem.
\qed\end{proof}
We make the following remark on Theorem \ref{maintheo}.
\begin{remark}
Without any additional assumptions on $\cG$, one can observe that the solution set of problem \eqref{probmulti} is unbounded
and the sequence $\{z^k\}$ generated by Algorithm \ref{alg:admm} may also be unbounded. Fortunately, we are still able to show in Theorem \ref{maintheo}(a) and (c) that the sequence $\{\big(x^k,y^k,\Pi_{\range(\cG)}(z^k)\big)\}$ converges to a solution to the KKT system \eqref{kkt}, and both the objective and the feasibility converge to the optimal value and zero, respectively.
Meanwhile, we would like to emphasize that the surjectivity assumption on $\cG$
 in Theorem \ref{maintheo}(e) is not restrictive at all.
 Indeed, this assumption simply means
that there are no redundant equations in the linear constraints $\cG x = b$ in the primal problem \eqref{modelexample}. If necessary, well established numerical linear algebra techniques can be used to remove redundant equations from $\cG x = b$.
\end{remark}

\subsection{The Two-Block Case}
\label{sect41}
Consider the two-block case that $\dY=\dY_1$ and $f$ is vacuous, i.e., the following problem
\[
\label{prob2b}
\min_{y,z}\left\{p(y)-\langle b,z\rangle\,|\,\cF^*y+\cG^* z=c\right\}.
\]
Assume that the KKT system of problem \eqref{prob2b} admits a nonempty solution set $\bK$.
For such a two-block problem,
Algorithm \ref{alg:admm} without the proximal terms and the inexact computations reduces to the classic ADMM.
Then, by Theorem \ref{maintheo}, the sequence $\left\{\big(x^k, y^k,\Pi_{\range{(\cG)}}(z^k)\big)\right\}$ generated by the classic ADMM or its inexact variants with $\tau\in(0,2)$ (in the order that the $y$-subproblem is solved before the $z$-subproblem) converges to {a point}
in $\bK$ if either $\cF$ is surjective or $p$ is strongly convex.
Moreover, if $\cG$ is also surjective, we have that the sequence $\left\{\big(x^k, y^k,z^k\big)\right\}$ converges to a point in $\bK$.
Note that the assumptions we made for problem \eqref{prob2b}
are {apparently} weaker than those in \cite{gabay1976}, where $\cF$ is assumed to be the identity operator, $\cG$ is surjective, and $p$ is assumed to be strongly convex.
Moreover, in \cite[Theorem 3.1]{gabay1976}, only the convergence of the primal sequence $\{(y^k,z^k)\}$ and the boundedness of the dual sequence $\{x^k\}$ were obtained.

The detailed comparison between the results in this paper and those in \cite{gabay1976} is presented in Table \ref{tablecompare}.
As can be observed from this table, the convergence result on the dual sequence $\{x^k\}$ is easier to be derived than that of the primal sequence $\{(y^k,z^k)\}$, and this result is consistent with the results in \cite{chennote} for the classic ADMM and the ALM in \cite{roc76b}.
Hence, the results derived in this paper properly resolves the questions we have mentioned in the introduction.

At last, we should mention that, in Sun et al. \cite[Theorem 3.3 (iv)]{sty2015}, a similar result to ours has been derived with the requirements that the initial multiplier $x^0$ satisfies $\cG x^0-b=0$ and all the subproblems are solved exactly. Here, we
are able to relax these requirements to the most
general case and extend our results to the more interesting and challenging multi-block problems.

\begin{table}
\caption{Comparison between \cite{gabay1976} and this paper. In the table `$\bf SOL$' denotes the solution set to problem \eqref{prob2b}, `{\bf X}' denotes the set of multipliers (the solution set to the dual problem) to problem \eqref{prob2b}, and
`$\bK$' denotes the solution set to the KKT system or problem \eqref{prob2b}, i.e., $\bK={\bf X}\times{\bf SOL}$.
The symbol $\to$ means that the sequence on its left-hand-side is convergent, and converges to a point in its right-hand-side.}
\label{tablecompare}
\small
\begin{center}
\begin{tabular}{|c|c|c|c|}
\hline&&\mc{2}{|c|}{}\\[-3mm]
{\bf Item} $\backslash$ {\bf Ref}
& \cite{gabay1976}
& \mc{2}{|c|}{This paper}
\\[1mm]
\hline&&\mc{2}{|c|}{}
\\[-2mm]
{\bf Updating rules}
& $z\Rightarrow y\Rightarrow x$\quad$\&$\quad $\tau\in(0,2)$
& \mc{2}{|c|}{$y\Rightarrow z\Rightarrow x$\quad$\&$\quad $\tau\in(0,2)$}
\\[1.5mm]
\hline&&&
\\[-2mm]
\multirow{2}{*}{\bf Assumptions-$y$}
& $p$ strongly convex
& $p$ strongly convex
& $p$ strongly convex
\\[3mm]
& {\color{blue} \bf and} $\cF$ the identity operator
& {\color{blue}\bf or} $\cF$ surjective
& {\color{blue}\bf or} $\cF$ surjective
\\[3mm]
\hline&&&
\\[0.5mm]
{\bf Assumptions-$z$}
&
$\cG$ surjective & - & $\cG$ surjective
\\[3mm]
\hline&&&
\\[-1.5mm]
\multirow{2}{*}{\bf Sequences}
& $\{(y^k,z^k)\}\to {\bf SOL}$
& $\dist\Big((x^k,y^k,z^k),\bK\Big)\to 0$
& $\{(y^k,z^k)\}\to {\bf SOL}$
\\[2mm]
& $\{x^k\}$ bounded
& $\{x^k\}\to{\bf X}$
& $\{x^k\}\to{\bf X}$
\\[3mm]
\hline
\end{tabular}
\end{center}
\end{table}
\normalsize

\subsection{Linear Rate of Convergence and Iteration Complexity}
Theorem \ref{ratethm} has provided a tool, which can be used together with Theorem \ref{theomain} to analyze the linear convergence rate of the sequence generated by Algorithm \ref{alg:admm}, i.e., one only need to verify whether \eqref{dkerr} is valid for this sequence, provided that the metric subregular property \eqref{msr} holds.
However, such a verification is not as straightforward as it conceptually seems.

Here, we establish a linear convergence result for the case that the linear system
in step 3 of Algorithm \ref{alg:admm} is solve exactly, but leave the general cases as a topic for further study.
For this purpose, we view problem \eqref{probmulti} as an instance of \eqref{prob1} with
\[
\label{defnewvariable}
\left\{
\begin{array}{l}
\varphi(w):=p(y_1),
\\[3mm]
h(w):=f(y)-\langle b,z\rangle,
\\[3mm]
\cA^* w:=\cF^*y+\cG^* z,
\end{array}\right.\quad \forall w=(y,z)\in\dW:=\dY\times\dZ.
\]
Then, the corresponding KKT residual mapping of problem \eqref{probmulti} can be given by \eqref{residual}.
Moreover, the self-adjoint linear operator $\Omega$ defined in \eqref{omega} is given by
$
\Omega(x;(y;z))=(x;\Theta^{\frac{1}{2}} (y;z))
$,
where
$
\Theta=\tau\sigma(\cP+\cT+\frac{(2-\tau)\sigma}{3}\cA\cA^*)
$
with $\cT$ and $\cP$ being defined in \eqref{tyz} and \eqref{hpo}, respectively.
In fact, we further have that
\[
\label{eq:omegarangez}
\Omega\big(x;(y;z)\big)
= \Omega\left(x;\big(y;\Pi_{\range{(\cG)}}(z)\big)\right), \quad \forall (x,y,z)\in \dX\times \dY\times \dZ.
\]

\begin{theorem}
\label{ratethmsp}
Suppose that Assumptions \ref{ass} and \ref{assp} hold.
Let $\{u^k=(x^k,w^k)\}$ with $w^k:=(y^k; z^k)$ be the sequence generated by Algorithm \ref{alg:admm} such that
$\{v^k := \big(x^k,y^k,\Pi_{\range{(\cG)}}(z^k)\big)\}$ converges to $v^*\in \bK$.
It holds that
\[
\label{eq:distudistv}
\dist_\Omega(u^k,\bK) = \dist_\Omega(v^k,\bK),\quad \forall k\ge 0.
\]
Suppose that $b-\cG x^{0}=0$ and $\gamma^k=0$ for all $k\ge 0$.
Suppose that the KKT residual mapping $\cR$ defined in \eqref{residual} (with the notation in \eqref{defnewvariable}) is metric subregular at $v^*$ for $0\in\dU$ with the modulus $\kappa>0$, in the sense that
there exists a constant $r>0$ such that \eqref{msr} holds with $\overline u=v^*$.
{Let $\{\tilde{\eta}_k\}$ be a given sequence of nonnegative numbers that converges to $0$ in the limit.
Suppose that in addition to
satisfying $\max\{\norm{\tilde{\delta}_i^k}, \norm{\delta_i^k}\mid i=1,\ldots,s\} \leq \tilde{\varepsilon}_k$,}
there exists an integer $k_0>0$ such that for any $k\ge k_0$, it holds that
\[
\label{errorss}
\max_{1\le i\le s}\left\{\norm{\tilde{\delta}_i^k}, \norm{{\delta}_i^k} \right\} \; \le
 \; \tilde \eta_k\|v^k - v^{k+1}\|.
\]
Then, for all $k$ sufficiently large, it holds that
$\dist_\Omega(u^{k+1},\bK)
\le\vartheta^k\, \dist_\Omega(u^k,\bK)$
with $\sup\limits_{k\geq k_0} \{ \vartheta_k\} < 1$, i.e., the convergence rate of $\dist_\Omega(u^k,\bK)$ is Q-linear when $k$ is sufficiently large.
\end{theorem}

\begin{proof}
By \eqref{eq:omegarangez}, we have that for all $k\ge 0$,
\begin{align*}
\dist^2_{\Omega}(u^k,\bK) ={}& \inf_{u\in \bK} \frac{1}{2} \inprod{u^k - u}{\Omega(u^k - u)} = \inf_{u\in \bK} \frac{1}{2} \inprod{u^k - u}{\Omega(v^k) - \Omega(u)} \\[5pt]
={}& \inf_{u\in \bK} \frac{1}{2} \inprod{v^k - u}{\Omega(v^k - u)}
= \dist^2_{\Omega}(v^k,\bK),
\end{align*}
i.e., \eqref{eq:distudistv} holds.
Since $b-\cG x^{0}=0$ and $\gamma^k=0$ for all $k\ge 0$,
according to \eqref{bmgs} one has that
$$
\|b-\cG x^{k}\|\le|1-\tau|^{k}\,\|b-\cG x^{0}\|+\tau\sum_{i=1}^{k}|1-\tau|^{k-i}\,\|\gamma^{i-1}\|
=0,
\quad\forall k\ge 0.
$$
Therefore, by \eqref{delta} and \eqref{hdelta} one knows that
$$
\Delta^k:=\delta_\sgs^{k}-\cF\cG^*
 (\cG\cG^*)^{-1}
\cG (x^{k}-x^{k-1})
=\delta
^{k}+\cN_u\cN_d^{-1}(\delta^k-\tilde\delta^k).
$$
Thus, we can get that for all $k\ge 0$,
$$
\begin{array}{ll}
\|d^k\|=\|\Delta^k\|
&\ds
\le\; (1+2\|\cN_u\cN_d^{-1}\|)
\max\{\norm{\tilde{\delta}},\norm{\tilde{\delta}}\}
\\[3mm]
&\ds\le\; \sqrt{s}\left(1+2\|\cN_u\cN_d^{-1}\|\right)\tilde\eta_k\|v^k-v^{k+1}\|,
\end{array}
$$
where $d^k:=(\what\Delta^k;\gamma^k)\in\dW$.
Define $\eta_k=\sqrt{s}(1+2\|\cN_u\cN_d^{-1}\|)\tilde\eta_k$.
Then, it holds that $\eta_k\to 0$ and $\|d^k\|\le\eta^k\|v^k-v^{k+1}\|$.
Therefore, by Theorem \ref{ratethm}, we know that for all $k$ sufficiently large
$$
\dist_\Omega(v^k,\bK)
\le\vartheta^k\, \dist_\Omega(v^k,\bK)
$$
with $\sup\limits_{k\geq k_0} \{ \vartheta_k\} < 1$, which, together with \eqref{eq:distudistv}, implies
$$
\dist_\Omega(u^k,\bK)
\le\vartheta^k\, \dist_\Omega(u^k,\bK) \quad \mbox{for all $k$ sufficiently large.}
$$
This completes the proof.
\qed
\end{proof}

\begin{remark}
Note that, different from the condition \eqref{dkerr} in Assumption \ref{assrate}, the condition \eqref{errorss} here is generally not directly verifiable during the numerical implementation.
However, Theorem \ref{ratethmsp} does provide us a very important theoretical guideline on implementing Algorithm \ref{alg:admm}, i.e., in the $k$-th iteration, it is likely to be beneficial to solve the subproblems to an accuracy higher than the dual feasibility $\|\cF^*y^k+\cG^*z^k-c\|$.
In fact, this phenomenon has already been observed during our numerical experiments.
We should also mention that even for the $2$-block case, the study on the linear convergence of inexact ADMMs with shorter step-length $\tau\in\big(0,\frac{1+\sqrt{5}}{2}\big)$ is still not as mature as the study for their exact counterparts, especially when compared with the recently developed results, e.g., in \cite{han,zhangning}.
Suitable criteria that generalize the condition \eqref{dkerr} for terminating the subproblems are still lacking.
We note that the results presented in Theorem \ref{ratethmsp} are still far from complete, and more effort should be spent on this part in the future.
\end{remark}

Finally, different from the above discussions on the convergence rate, we can establish
the following non-ergodic iteration complexity for the sequence generated by Algorithm \ref{alg:admm}
by a direct application of Theorem \ref{theomain}.

\begin{theorem}
Suppose that Assumptions \ref{assp} and \ref{ass} hold.
Let $\{u^k=(x^k,w^k)\}$ with $w^k:=(y^k; z^k)$ be the sequence generated by Algorithm \ref{alg:admm} such that
$\{v^k := \big(x^k,y^k,\Pi_{\range{(\cG)}}(z^k)\big)\}$ converges to $v^*\in \bK$.
It holds that the KKT residual \eqref{residual}, with $\cB$ and $\cP$ given by \eqref{hpo}, satisfies
$$
\min_{0\le j\le k}\|\cR(u^{j})\|^2\le \varrho/k,
\quad\mbox{and}
\quad
\lim_{k\to\infty}\{k\cdot\min_{0\le j\le k}\|\cR(u^{j})\|^2\}=0,
$$
where the constant $\varrho$ is defined as in \eqref{varrho} but with
$$
\begin{array}{l}
{\bf e} :=\|u^0_e\|^2_\Omega+2\tau\sigma\|\Theta^{-1/2}\|\left(\sum_{j=0}^\infty\tilde\varepsilon_j\right)
\left(\|u_e^0\|_\Omega+\mu \sum_{j=0}^\infty\tilde\varepsilon_j\right)+4
\left(\sum_{j=1}^{\infty}\tilde\varepsilon_j^2\right)
\end{array}
$$
and $u^0_e = u^0 - v^*$.
\end{theorem}
\begin{proof}
From \eqref{eq:omegarangez}, we know that
\[
\label{eq:ue0ve0}
\norm{u^0 - v^*}^2_{\Omega} = \inprod{\Omega(u^0 - v^*)}{u^0 - v^*}
= \inprod{\Omega(v^0) - \Omega(v^*)}{u^0 - v^*}
= \norm{v^0 - v^*}^2_{\Omega}.
\]
According to \eqref{residual}, \eqref{kkt} and \eqref{defnewvariable}, one has that
$$
\cR(u)=\begin{pmatrix}
c-\cF^*y+\cG^* z
\\[1mm]
y-\prox_{P}
(y-\nabla f(y)-\cF x)
\\[1mm]
\cG x-b
\end{pmatrix}, \quad\forall u=(x,y,z)\in\dX\times\dY\times\dZ.
$$
Since for all $k\ge 0$,
$\cG^*z^k=\cG\Pi_{\range(\cG^*)}(z^k)$, one has that $ \cR(u^k)=\cR(v^k)$.
Therefore, by using \eqref{proxixx} in Theorem \ref{theomain}(a), Theorem \ref{thmiter} and \eqref{eq:ue0ve0}, one has the results of this theorem holds. This completes the proof.
\qed\end{proof}

\section{Numerical Experiments}
\label{sec:num}
In this section, we conduct numerical experiments on solving dual linear SDP and dual convex quadratic SDP problems via Algorithm \ref{alg:admm}
with the dual step-length $\tau$ taking values beyond the
standard restriction of $(1+\sqrt{5})/2$.
For linear SDP problems, the algorithm reduces to the two-block ADMM, and the aim is two-fold.
On the one hand, as the ADMM is among the most important first-order algorithms for solving
SDP problems, it is of importance to know to what extent can the numerical efficiency be improved if
the observation on the dual step-length made in this paper is incorporated.
On the other hand, as the upper bound of the step-length has been enlarged, it is also important
to see whether a step-length that is very close to the upper bound
will lead to better or worse numerical performance.

A standard linear SDP problem has the following form:
\[
\label{numsdp}
\min_{X}\{ \langle C,X\rangle \; |\; \bA X={\bf b}, X\in\S_+^n\}
\]
and its corresponding dual is given as in \eqref{probsdp}.
To avoid repetition, we refer the reader to \eqref{probsdp}
for the notation used.
The (majorized) augmented Lagrangian function associated with {problem} \eqref{probsdp} is given by
$$
\begin{array}{r}
\cL_{\sigma}(S,{\bf z};X) \;=\; \delta_{\S_+^n}(S)-\langle {\bf b},{\bf z}\rangle
+\langle X,S+\bA^*{\bf z}-C\rangle
+\frac{\sigma}{2}\|S+\bA^*{\bf z}-C\|^2,
\qquad
\\[3mm]
\forall (S,{\bf z},X)\in\S^n\times\R^m\times\S^n,
\end{array}
$$
where $\sigma>0$ is the given penalty parameter.
When applied to solving problem {problem} \eqref{probsdp}, at the $k$-th step of the two-block ADMM the following steps are performed:
$$
\left\{
\begin{array}{l}
S^{k+1}=\Pi_{\S_+^n}(C-\bA^*{\bf z}^k-X^k/\sigma),
\\[3mm]
{\bf z}^{k+1}=(\bA\bA^*)^{-1}\left(\bA(C-S^{k+1})-(\bA X^k-{\bf b})/\sigma\right),
\\[3mm]
X^{k+1}=X^k+\tau\sigma(S^{k+1}+\bA^*{\bf z}^{k+1}-C),
\end{array}
\right.
$$
where the step-length $\tau$ is allowed to be in the range $(0,2)$ based on
Theorem \ref{theomain} and the discussions in Section \ref{sect41}.
We emphasize again that this is in contrast to the usual interval of
$(0,(1+\sqrt{5})/2)$ allowed by the convergence analysis of Glowinski in \cite[Theorem 5.1]{glo80}.

On the other hand, as was briefly introduced in Section \ref{secccqp}, the convex QSDP problem was formally given in \eqref{qsdp}, whose dual problem, in minimization form, is a multi-block problem given by
\[
\label{dqsd}
\begin{array}{cl}
\ds
\min_{S,W,{\bf z}_E,{\bf z}_I}
&
\delta_{\S_+^n}(S)
+{\delta_{\R_{+}^{m_I}}({\bf z}_I)}
+\frac{1}{2} \langle W,\mathbfcal{Q} W\rangle
-\langle {\bf b}_E,{\bf z}_E\rangle
-\langle {\bf b}_I,{\bf z}_I\rangle
\\[3mm]
\mbox{s.t.}
&
S
-
\mathbfcal{Q}W
+\bA_E^*{\bf z}_E+\bA_I^*{\bf z}_I
+C=0.
\end{array}
\]
Note that problem \eqref{qsdp} was subsumed as an instance of the convex quadratic composition optimization problem \eqref{probcqpie}.
Therefore, {to fit the framework of} Algorithm \ref{alg:admm}, we write the dual of problem \eqref{qsdp} in the minimization form as
\[
\label{dqsdpnum}
\begin{array}{cl}
\ds
\min_{S,W,{\bf s},{\bf z}_E,{\bf z}_I}
&
\delta_{\S_+^n}(S)
+\delta_{\R_{+}^{m_I}}({\bf s})
+\frac{1}{2} \langle W,\mathbfcal{Q} W\rangle
-\langle {\bf b}_E,{\bf z}_E\rangle
-\langle {\bf b}_I,{\bf z}_I\rangle
\\[3mm]
\mbox{s.t.}
&
\left\{
\begin{array}{l}
S
-
\mathbfcal{Q}W
+\bA_E^*{\bf z}_E+\bA_I^*{\bf z}_I
+C=0,
\\[3mm]
D({\bf s}-{\bf z}_I)=0,
\end{array}
\right.
\end{array}
\]
where $D\in\R^{m_I\times m_I}$ is a {given} positive definite diagonal matrix which is
incorporated here for for the purpose of scaling the variables to ensure the numerical stability.

The convex QSDP problem \eqref{qsdp} is solved via its dual \eqref{dqsdpnum}, whose (majorized) augmented Lagrangian function is defined by
$$
\begin{array}{ll}
\cL_\sigma(S,W,{\bf z}_E,{\bf z}_I,s;X,{\bf x}):
=&\big( \delta_{\S^n_{+}}(S)+\delta_{\R_{+}^{m_I}}({\bf s})\big)
+\frac{1}{2}\inprod{W}{\mathbfcal{Q} W}
-\inprod{{\bf b}_E}{{\bf z}_E}
-\inprod{{\bf b}_I}{{\bf z}_I}
\\[3mm]
&+\langle X, S-\mathbfcal{Q} W +\bA_E^* {\bf z}_E + \bA_I^*{\bf z}_I + C\rangle
+\langle D({\bf s} - {\bf z}_I),{\bf x}\rangle
\\[3mm]
&+\frac{\sigma}{2}\| S-\mathbfcal{Q} W +\bA_E^* {\bf z}_E + \bA_I^*{\bf z}_I + C\|^2
+\frac{\sigma}{2}\| D({\bf s} - {\bf z}_I)\|^2,
\\[2mm]
&\quad\ds
\forall (S,W,{\bf z}_E,{\bf z}_I,s;X,{\bf x})\in\S^n\times \S^n\times\R^{m_E}\times\R^{m_I}\times\R^{m_I}\times \S^n\times\R^{m_I}. 
\end{array}
$$
where $\sigma>0$ is the given penalty parameter and
and we have used $X\in\S^{n}$ and ${\bf x}\in\R^{m_{I}}$ to denote the Lagrange multipliers which are introduced for the two groups of equality constraints in \eqref{dqsdpnum}.
During the $k$-th iteration of Algorithm \ref{alg:admm} with given $(S^k,W^k,{\bf z}^k_E,{\bf z}^k_I,{\bf s}^k)$ and $(X^k,{\bf x}^k)$, we update the variables in the order of
$$
\Big(
\underbrace{{\bf z}_E^{k+1/2}
\Rightarrow W^{k+1/2}
}_{\mbox{backward GS}}
\Rightarrow
\underbrace{ (S^{k+1},{\bf s}^{k+1}) \Rightarrow W^{k+1}
\Rightarrow {\bf z}^{k+1}_E}_{\mbox{forward GS}}
\Big)
\Rightarrow {\bf z}_I^{k+1}
\
\Rightarrow \underbrace{\Big(X^{k+1},{\bf x}^{k+1}\Big)}_{\tau\in(0,2)}.
$$
Note that the term $\inprod{{\bf b}_I}{{\bf z}_I}$ is treated as the linear term
in the framework of \eqref{probmulti}. We made this choice because for the
test instances that we will consider later, the linear system that must be solved
to update ${\bf z}_I$ is much larger than that for updating ${\bf z}_E$,
and in this way, the larger linear system will be solved only once in each iteration.

The numerical results in the subsequent two subsections are obtained by using {\sc Matlab} R2017b on a HP Elitedesk (64-bit Windows 10 system) with one Intel Core i7-4770S Processor (4 Cores, $3.1-3.9$ GHz) and $16$ GB RAM (with the virtual memory turned off).

\subsection{Numerical Results on Linear SDP Problems}
\label{numsec1}
Based on the first-order optimality condition for problem \eqref{numsdp}, we terminate all the tested algorithms if
$$
\begin{array}{c}
\eta_{\textup{SDP}} :=
\max\{\eta_D, \eta_P, \eta_{S}\} \;\leq \; 10^{-6},
\end{array}
$$
where
$$
\left\{
\begin{array}{lll}
\ds
\eta_{D}=\frac{\|\bA^{*}{\bf z}+S-C\|}{1+\|C\|},\
\eta_{P}=\frac{\|\bA X-{\bf b}\|}{1+\|{\bf b}\|},
\\[5mm]
\ds
\eta_{S}=\max\left\{\frac{\|X-\Pi_{\S_{+}^{n}}(X)\|}{1+\|X\|},
\frac{|\langle X,S\rangle|}{1+\|X\|+\|S\|}\right\}
\end{array}
\right.
$$
with the maximum number of iterations set at $100,000$.
In addition, we also measure the duality gap:
$$
\eta_{\rm gap}:=\frac{ \langle C,X\rangle-\langle {\bf b},{\bf z}\rangle}{1+|\langle C,X\rangle|+|\langle {\bf b},{\bf z}\rangle|}.
$$
During our preliminary tests, we found that using a step-length smaller than $1$ is not as good as using the unit step-length. Therefore, we shall only consider the cases that $\tau\ge 1$.
Note that the known theoretical upper bound of the step-length $\tau$ in the classic ADMM for solving general convex programming is $\frac{1+\sqrt{5}}{2} \approx 1.618034$.
Although it has been observed empirically that the ADMM with the step-length $\tau = 1.618$ works quite well, this phenomenon still requires further understanding since the value $1.618$ is quite close to the theoretical upper bound and such an aggressive choice may result in unstable
numerical performance.
Fortunately, the above concern is partially alleviated by the theoretical results obtained in this paper. Indeed, for a large class convex optimization problems, one can use $\tau=1.618$ confidently since it has a ``safe'' distance to the renewed theoretical upper bound of $2$.
For this class of problems, it is thus very interesting to see what
would happen if the step-length $\tau$ is very close to $2$. Therefore, we tested five choices of the step-length, i.e., $\tau=1$, $1.618$, $1.90$, $1.99$ and $1.999$.
For convenience, we use ADMM($\tau$) to denote the algorithm with the specific step-length $\tau$.

We tested $6$ categories of linear SDP problems, including the random sparse SDP problems tested in \cite{malick}, the semidefinite relaxation of frequency assignment problems
(FAP) \cite{eisnbl}, the relaxation of maximum stable set problems \cite{toh04,trick,sloane}, the SDP relaxation of binary integer quadratic (BIQ) problems from \cite{wiegele},
the SDP relaxation of rank-1 tensor approximations (R1TA) \cite{nie12,nie14},
and the SDP relaxations of clustering problems \cite{peng07}. One may refer to \cite{zhao,yst2015} for detailed descriptions and the data sources of these problems.
All these algorithms are tested by running the {\sc Matlab} package SDPNAL$+$ (version 1.0, available at \url{http://www.math.nus.edu.sg/~mattohkc/SDPNALplus.html}).
The records of the computational results are provided in Table \ref{tablesdpresult}.
Here, we should mention that even though all the problems we tested have been successfully solved by at least one of the tested algorithms, there are a few categories of SDP problems that are beyond the capability of the ADMM, see, e.g., \cite{zhao}.

\begin{figure}
 \centering
 \includegraphics[width=0.9\textwidth]{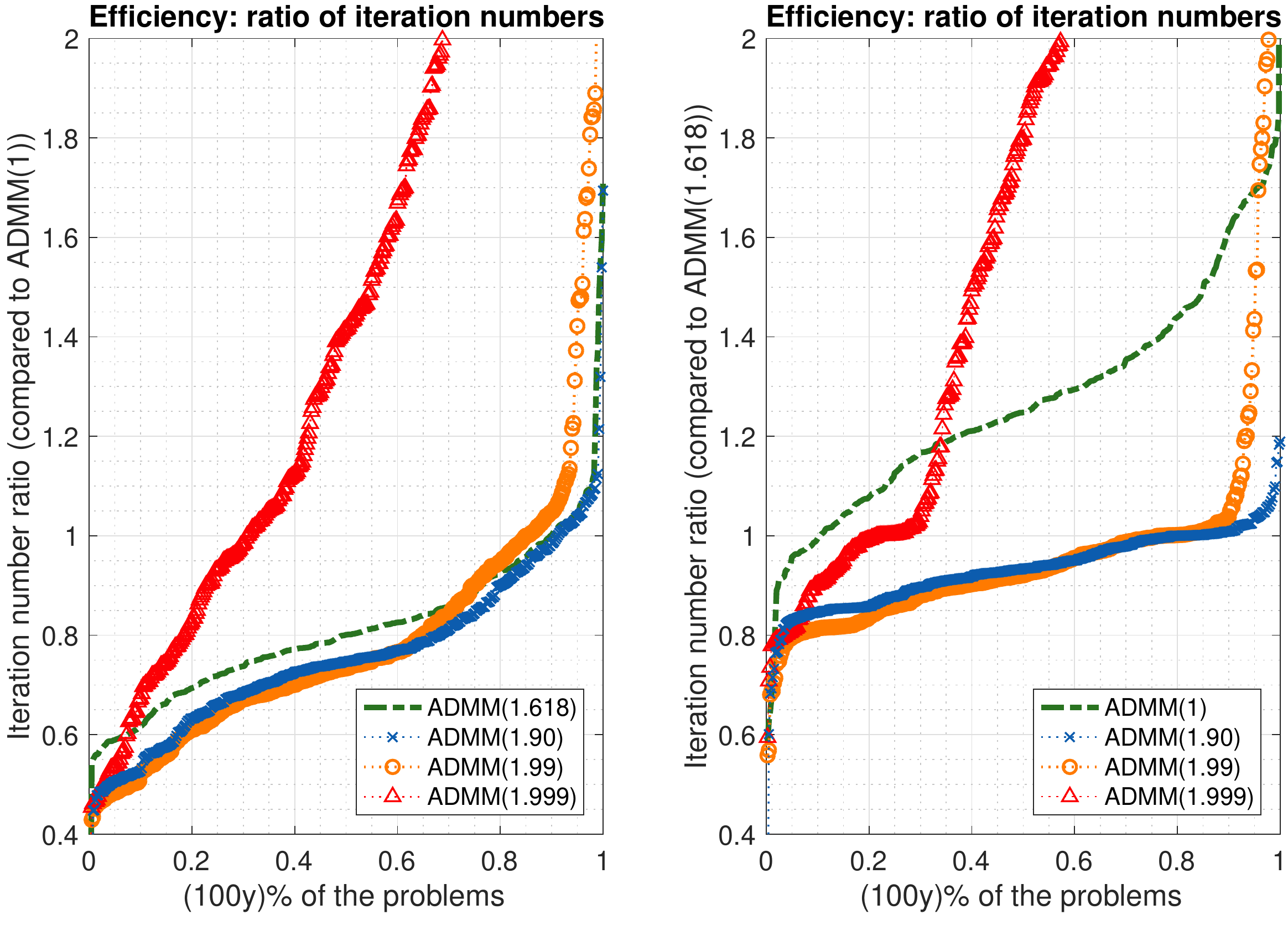}
 \caption{Comparison of the computational efficiency of the classic two-block ADMM with different step-lengths}
 \label{figeff}
\end{figure}

Figure \ref{figeff} presents the computational performance of the ADMM with all the five choices of step-lengths.
The left {panel} shows the comparison between ADMM($1$) and all the other algorithms, while the right {panel} shows the comparison between ADMM($1.618$) and all the others.
As can be seen from Figure \ref{figeff}, ADMM($1.618$) has an impressive improvement over ADMM($1$) and ADMM($1.9$) works even better than ADMM($1.618$) for more than $80\%$ of the tested instances.
Furthermore, ADMM($1.99$) can perform marginally better than ADMM($1.9$) for about $60\%$ of the tested instances but for about 10\% of them, its performance is apparently worse.
However, ADMM($1.999$) has a significantly worse performance than ADMM($1.99$)
even though its step-length is just slightly larger than $1.99$. This can be partially explained by the fact that the step-length of $1.999$ is too close to the theoretical upper bound of $2$.

From both the theoretical analysis and numerical experiments in this paper, one can see that in general it is a good idea to use a step-length that is larger than $1$, e.g., $\tau=1.618$, when solving linear SDP problems.
Meanwhile, we can even set the step-length to be larger than $1.618$, say $\tau=1.9$, to get even better numerical performance.

\subsection{Numerical Results on Convex QSDP Problems}
The KKT system of problem \eqref{dqsd} is given by
\[
\label{eq:op-M}
\left\{
\begin{array}{l}
\ds
S-\mathbfcal{Q}W+\bA_E^*{\bf z}_E+\bA_I^*{\bf z}_I+C=0,\
\bA_{E}X-{\bf b}_{E}=0,
\\[3mm]
\ds
\mathbfcal{Q} X-\mathbfcal{Q} W=0, \
X\in\S^{n}_{+},\
S\in\S_{+}^{n},\
\langle X, S\rangle=0,
\\[3mm]
\ds
\bA_{I}X-{\bf b}_{I}\ge0,\
{\bf z}_{I}\ge 0,\
\langle \bA_{I}X-{\bf b},{\bf z}_{I}\rangle =0.
\end{array}
\right.
\]
Based on the optimality conditions given in \eqref{eq:op-M}, 
we measure the accuracy of an approximate (computed) solution $(X,Z,W,S,y_E,y_{I})$ for the convex
 QSDP \eqref{qsdp} and its dual \eqref{dqsd} via
\begin{equation}
\label{stop:sqsdp}
\begin{array}{c}
\eta_{\textup{qsdp}} =
\max\left\{\eta_D, \eta_P,\eta_{W}, \eta_{S},
\eta_{I}\right\},
\end{array}
\end{equation}
where
$$
\left\{
\begin{array}{l}
\ds
\eta_{D}=\frac{\|S-\mathbfcal{Q}W+\bA_E^*{\bf z}_E+\bA_I^*{\bf z}_I+C\|}{1+\|C\|},\quad
\eta_{P}=\frac{\|\bA_{E}X-{\bf b}_{E}\|}{1+\|{\bf b}_{E}\|},
\\[5mm]
\ds
\eta_{W}=\frac{\|\mathbfcal{Q} X-\mathbfcal{Q} W\|}{1+\|Q\|},\quad
\eta_{S}=\max\left\{\frac{\|X-\Pi_{\S_{+}^{n}}(X)\|}{1+\|X\|},
\frac{|\langle X,S\rangle|}{1+\|X\|+\|S\|}
\right\},
\\[5mm]
\ds
\eta_{I}=\max\left\{
\frac{\|\min(0,\,{\bf z}_{I})\|}{1+\|{\bf z}_{I}\|},
\frac{\|\min(0,\,\bA_{I}X-{\bf b}_{I})\|}{1+\|{\bf b}_{I}\|},
\frac{|\langle \bA_{I}X-{\bf b}_I,{\bf z}_{I}\rangle|}{1+\|\bA_{I}x-{\bf b}_I\|+\|{\bf z}_{I}\|}\right\}.
\end{array}
\right.
$$
Additionally, we measure the objective values and the duality gap:
$$
\eta_{\rm gap}:=\frac{\rm Obj_{primal}-Obj_{dual}}{1+|\rm Obj_{primal}|+|Obj_{dual}|},
$$
where 
$$
{{\rm Obj}_{\rm primal}}:=\frac{1}{2}\langle X,\mathbfcal Q X\rangle-\langle C,X\rangle,
\quad\mbox{and}\quad
{\rm Obj_{dual}}:
=-\frac{1}{2}\inprod{W}{\mathbfcal{Q} W}
+ \inprod{{\bf b}_E}{{\bf z}_E} + \inprod{{\bf b}_I}{{\bf z}_I}.
$$
In our numerical experiments, similar to \cite{chenl2015}, we used QSDP test instances
based on the SDP problems arising from the relaxation of the binary integer quadratic (BIQ) programming with a large number of inequality constraints that was introduced by Sun et al. \cite{sty2015} for getting tighter bounds.
The problems that we actually solve have the following form:
$$
\begin{array}{cl}
\min &\ds
\frac{1}{2}\inprod{X}{\mathbfcal{Q} X} + \frac{1}{2} \inprod{Q}{\overline{X}} + \inprod{c}{x} \\[3mm]
\mbox{s.t.}&\left\{
\begin{array}{l}
 \textup{diag}(\overline{X}) - x = 0, \
 X = \begin{pmatrix}
\overline{X} & x \\
x^T & 1 \\
\end{pmatrix}
\in \Sn_+,
\\[5mm]
x_i-\overline{X}_{ij}\geq 0,\,
 x_j-\overline{X}_{ij} \geq 0,\,
 \overline{X}_{ij} - x_i -x_j \geq -1,\
\forall\, 1\leq i < j \leq n-1.
\end{array}
\right.
\end{array}
$$
The data for $Q$ and $c$ are taken from the Biq Mac Library maintained by Wiegele \cite{wiegele}.

We solve the QSDP \eqref{qsdp} via its dual \eqref{dqsdpnum} with the matrix {$D=({\sqrt{\|\bA_{I}\|}}/{2}) I_{\R^{m_I}}$.}
We use the directly extended multi-block {ADMM} with step-length $\tau=1$ as the benchmark (which we named as `{\tt Directly Extended}'), and compare it with Algorithm \ref{alg:admm}, which was implemented in $6$ different ways, i.e., $2$ groups of algorithms with each using $3$ types of different step-lengths, i.e., $\tau=1$, $1.618$ and $1.9$, which are chosen according to the numerical results in Section \ref{numsec1}.
For convenience, we use the name `{\tt sGS-PADMM}' to mean that Algorithm \ref{alg:admm} is implemented
{such that all the subproblems are solved exactly via direct solvers or adding appropriate} proximal terms, and use `{\tt sGS-iPADMM}' to mean that Algorithm \ref{alg:admm} is implemented
{such that the subproblems are allowed to be solved inexactly via iterative solvers.}
The details of all the {seven} tested algorithms are presented in Table \ref{algorithms}.
\begin{table}[H]
\caption{$7$ types of algorithms tested for the convex QSDP problems. In the table,
`{\tt Dir}' denotes the corresponding subproblems are solved via direct solvers, `{\tt Proj.}'
means that the corresponding subproblems are the calculation of projections,
`{\tt Prox.}' means that the corresponding subproblems are solved via adding appropriate proximal terms to ensure closed-form solutions,
`{\tt Inex.}' means that the subproblems are solved approximately via an iterative scheme,
`{\tt (rep.)}' means that the corresponding subproblems in the forward sGS sweeps are also solved,
`{\tt (check)}' means that the corresponding subproblems in the forward sGS sweeps are not directly solved but the most recently updated variables are checked to see if they are admissible approximate solution to the subproblems.
}\label{algorithms}
\small
\begin{center}
\begin{tabular}{|c|c|c|c|c|c|c|}
\hline
&&&&&&
\\[-2mm]
{\bf Algorithm} $\backslash$ {\bf Variable}
& $\tau$
& ${\bf z}_E$
& $W$
& $S$
& ${\bf s}$
& ${\bf z}_I$
\\[1mm]
\hline
&&&&&&
\\[-2mm]
{\tt Directly Extended}
& $1$
& {\tt Dir.}
& {\tt Prox.}
& {\tt Proj.}
& {\tt Proj.}
& {\tt Prox.}
\\
\hline
&&&&&&
\\[-2mm]
\multirow{4}{*}{\tt sGS-PADMM}
& $1$
& {\tt Dir.(rep.)}
& {\tt Prox.(rep.)}
& {\tt Proj.}
& {\tt Proj.}
& {\tt Prox.} \\[1mm]
\cline{2-7}
&&&&&&
\\[-2mm]
& $1.618$
& {\tt Dir.(rep.)}
& {\tt Prox.(rep.)}
& {\tt Proj.}
& {\tt Proj.}
& {\tt Prox.} \\[1mm]
\cline{2-7}
&&&&&&
\\[-2mm]
& $1.9$
& {\tt Dir.(rep.)}
& {\tt Prox.(rep.)}
& {\tt Proj.}
& {\tt Proj.}
& {\tt Prox.} \\[1mm]
\hline
&&&&&&
\\[-2mm]
\multirow{4}{*}{\tt sGS-iPADMM}
& $1$
& {\tt Dir.\blue{(check)}}
& {\tt \red{Inex.}\blue{(check)}}
& {\tt Proj.}
& {\tt Proj.}
& {\tt \red{Inex.}}\\[1mm]
\cline{2-7}
&&&&&&
\\[-2mm]
& $1.618$
& {\tt Dir.\blue{(check)}}
& {\tt \red{Inex.}\blue{(check)}}
& {\tt Proj.}
& {\tt Proj.}
& {\tt \red{Inex.}}\\[1mm]
\cline{2-7}
&&&&&&
\\[-2mm]
& $1.9$
& {\tt Dir.\blue{(check)}}
& {\tt \red{Inex.}\blue{(check)}}
& {\tt Proj.}
& {\tt Proj.}
& {\tt \red{Inex.}}\\[1mm]
\hline
\end{tabular}
\end{center}
\end{table}
\normalsize
For all the algorithms applied to problem \eqref{dqsdpnum}, the subproblems corresponding to the block variable $(S,s)$ can {be}
solved analytically by computing the projections onto $\S^n_+\times \R^{m_I}_+$.
For the subproblems corresponding to ${\bf z}_E$, linear systems of equations must be solved with the same coefficient matrix
$\bA_E\bA_E^*$. As the linear systems are not too large, we solve them via the Cholesky factorization (computed only once) of $\bA_E\bA_E^*$.
 For the subproblems corresponding to ${\bf z}_{I}$ and $W$,
we need to solve very large scale linear systems of equations, so that they are {either} solved via a preconditioned conjugate gradient method with preconditioners that are described in \cite[Section 7.1]{chenl2015} (for {\tt sGS-iPADMM}), or solved directly by adding an appropriate proximal term to the subproblems to get closed-form solutions (for {\tt sGS-PADMM}).
Moreover, in the implementation of the {\tt sGS-PADMM}, all the subproblems {in} the forward Gauss-Seidel sweep are directly solved, while in the implementation of {\tt sGS-iPADMM} we {used the strategy described in \cite[Remark 4.1(b)]{chenl2015} to decide whether the computation of the subproblems in the forward GS sweep can be skipped} (see \cite[Section 7.2]{chenl2015} for more details on using this technique).
We used {a similar strategy as} described in \cite[Section 4.4]{lam} to adaptively adjust the penalty parameter $\sigma$
and {used} the same technique as in \cite{chenl2015} to control the error tolerance for solving the subproblems, i.e., $\{\tilde \varepsilon_{k}\}_{k\ge 0}$ is chosen such that $\alpha\tilde\varepsilon_{k}\le 1/k^{1.2}$, where $\alpha>0$ is a positive constant based on the problem data.

We have tested $147$ instances of convex QSDP problems with $n$ ranging from $51$ to $501$. The linear operator $\mathbfcal{Q}$ was chosen as the symmetrized Kronecker operator ${\mathbfcal Q}(X)=\frac{1}{2}(AXB+BXA)$ with $A$ and $B$ being two randomly generated symmetric positive semidefinite matrices such that ${\rm rank}(A)=10$ and ${\rm rank}(B)\approx n/5$, respectively, as was used in \cite{toh08,chenl2015}.
The maximum iteration number is set at $500,000$.
The detail computational results are provided in Table \ref{tableqsdpresult}.

\begin{figure}[H]
 \centering
\includegraphics[width=0.9\textwidth]{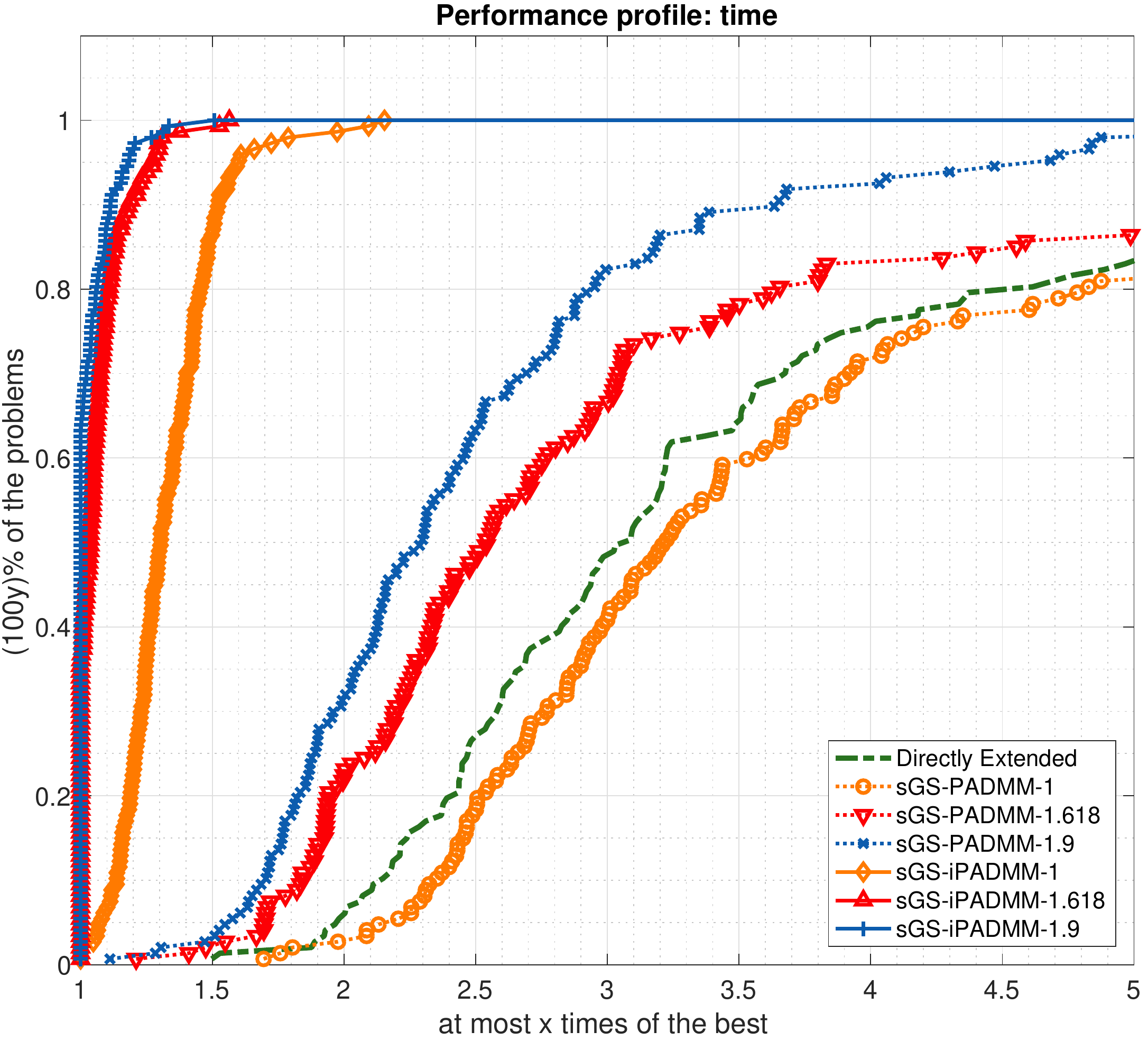}
\caption{Comparison of the computational efficiency of the $7$ algorithms}
\label{figtest}
\end{figure}

Figure \ref{figtest} shows the numerical performance of the $7$ tested algorithms described in Table \ref{algorithms} on solving the convex QSDP problems to the accuracy of $10^{-6}$ in terms of $\eta_{\textup{qsdp}}$ in \eqref{stop:sqsdp}.
One can readily see from the figure that
{\tt sGS-iPADMM} overwhelmingly outperforms {\tt sGS-PADMM}, no matter which step-length $\tau$ was used.
This evidently shows the considerable advantage of catering for approximate solutions in the subproblems of Algorithm \ref{alg:admm}. Moreover, for both {\tt sGS-PADMM} and {\tt sGS-iPADMM}, the step-length $\tau=1.618$ is able to bring a noticeable improvement on the numerical efficiency, compared with using the unit step-length. Meanwhile, the choice of $\tau=1.9$ can perform even better in general.
Even this is more apparent for {\tt sGS-PADMM}, in which all the subproblems are solved exactly.
{We can see} that {\tt sGS-iPADMM} with $\tau=1.9$ performs the best among all the tested algorithms for almost $65\%$ of all the tested problems.
Hence, the numerical results clearly demonstrate the merit of using a larger step-length and the flexibility of inexactly solving the subproblems.

\section{Conclusions}
In this paper, we have shown that, for a class of convex composite programming problems, the sequence generated by an inexact sGS decomposition based multi-block majorized (proximal) ADMM is equivalent to the sequence generated by an inexact proximal ALM starting with the same initial point. The convergence of the inexact majorized proximal ALM was first established, and the convergence of the multi-block ADMM-type algorithm follows readily because of the newly discovered equivalence.
As a consequence of this equivalence, we are able to provide a very general answer to the open question on whether the whole sequence generated by the classic ADMM with $\tau\in(0,2)$ for a conventional two-block problem with one part of its objective function being linear, is convergent.
Numerical experiments on solving a large number of linear and convex quadratic SDP problems are conducted.
The numerical results show that one can achieve even better numerical performance of the ADMM if the step-length is chosen to be larger than the conventional upper bound of $(1+\sqrt{5})/2$, and one can get a considerable improvement by allowing inexact subproblems together with the large step-lengths on the multi-block ADMM for convex quadratic SDP problems.
 We hope that our theoretical analysis and numerical results can inspire more insightful studies on the ADMM-type algorithms.

\section*{Acknowledgments}
We would like to thank the two anonymous referees for their careful reading of this paper, and their insightful comments and suggestions which have helped to improve the quality of this paper.

\small
\bibliographystyle{siamplain}

\newpage
\clearpage

\newgeometry{left=2.54cm,right=2.54cm,top=2cm,bottom=1.7cm}
\scriptsize

\begin{landscape}

\begin{center}
\begin{longtable}{|cc|c|c|c|c|}
\caption{ The performance of ADMM with $\tau=1,1.618,1.90,1.99,1.999$ on solving SDP problems (accuracy $= 10^{-6}$). The maximum iteration number is set as $100,000$.
In the table, the time of computation is in the format of "hours:minutes:seconds" if it is larger than one minute, or else it is recorded by second.}
\label{tablesdpresult}
\vspace{0.5em}\\
\hline
%%%%%%%%%%%%%%%%%%%%%%%%%%%% 1st line %%%%%%%%%%%%%%%%%%%%%%%%%%%%%%%%%%
\mc{2}{|c|}{}
&\mc{1}{c|}{}
&\mc{1}{c|}{}
&\mc{1}{c|}{}
&\mc{1}{c|}{}
\\[-5pt]
%%%%%%%%%%%%%%%%%%%%%%%%%%%% 2nd line %%%%%%%%%%%%%%%%%%%%%%%%%%%%%%%%%%
\mc{2}{|c|}{}
&\mc{1}{c|}{iteration}
&\mc{1}{c|}{$\eta_{\rm SDP}$}
&\mc{1}{c|}{$\eta{gap}$}
&\mc{1}{c|}{CPU time}
\\[2pt]
\hline
%%%%%%%%%%%%%%%%%%%%%%%%%%%% 3rd line %%%%%%%%%%%%%%%%%%%%%%%%%%%%%%%%%%
\mc{1}{|c}{problem}&\mc{1}{c|}{$m|n$}
&\mc{1}{c|}{$\tau=1|1.618|1.90|1.99|1.999$}
&\mc{1}{c|}{$\tau=1|1.618|1.90|1.99|1.999$}
&\mc{1}{c|}{$\tau=1|1.618|1.90|1.99|1.999$}
&\mc{1}{c|}{$\tau=1|1.618|1.90|1.99|1.999$}
\\
\hline
\endfirsthead
\multicolumn{3}{c}{Table~\thetable~(continued) }\vspace{0.5em}
\\
\hline
%%%%%%%%%%%%%%%%%%%%%%%%%%%%%%%%%%%%%%%%%%%%%%%%%%%%%%%%%%%%%%
\mc{2}{|c|}{}
&\mc{1}{c|}{}
&\mc{1}{c|}{}
&\mc{1}{c|}{}
&\mc{1}{c|}{}
\\[-5pt]
%%%%%%%%%%%%%%%%%%%%%%%%%%%% 2nd line %%%%%%%%%%%%%%%%%%%%%%%%%%%%%%%%%%
\mc{2}{|c|}{}
&\mc{1}{c|}{iteration}
&\mc{1}{c|}{$\eta_{\rm SDP}$}
&\mc{1}{c|}{$\eta{gap}$}
&\mc{1}{c|}{CPU time}
\\[2pt]
\hline
%%%%%%%%%%%%%%%%%%%%%%%%%%%% 3rd line %%%%%%%%%%%%%%%%%%%%%%%%%%%%%%%%%%
\mc{1}{|c}{problem}&\mc{1}{c|}{$m|n$}
&\mc{1}{c|}{$\tau=1|1.618|1.90|1.99|1.999$}
&\mc{1}{c|}{$\tau=1|1.618|1.90|1.99|1.999$}
&\mc{1}{c|}{$\tau=1|1.618|1.90|1.99|1.999$}
&\mc{1}{c|}{$\tau=1|1.618|1.90|1.99|1.999$}
\\
\hline
\endhead
\hline
\endfoot
\input{table-SDP.txt}
\hline
\end{longtable}
\end{center}

\newpage
\tiny

\begin{center}
\begin{longtable}{|cc|c|c|c|c|}
\caption{The performance of $7$ multi-block ADMM-type algorithms on solving convex quadratic SDP problems (accuracy $= 10^{-6}$). The maximum iteration number is set as $500,000$.
In the table, `D-E' means the directly extended multi-block ADMM with step-length $\tau=1$, P$1.0$ (P$1.6$, P$1.9$) means the block sGS decomposition based multi-block proximal ADMM with step-length $\tau=1$ ($\tau=1.618$, $\tau=1.9$, etc.), and $iP1.0$ ($iP1.6$, $iP1.9$, etc.) means the block sGS decomposition based inexact multi-block ADMM with step-length $\tau=1$ ($\tau=1.618$, $\tau=1.9$, etc.).
The time of computation is in the format of "hours:minutes:seconds" if it is larger than one minute, or else it is recorded by second.
}
\label{tableqsdpresult}
\vspace{0.5em}\\
\hline
%%%%%%%%%%%%%%%%%%%%%%%%%%%% 1st line %%%%%%%%%%%%%%%%%%%%%%%%%%%%%%%%%%
\mc{2}{|c|}{}
&\mc{1}{c|}{}
&\mc{1}{c|}{}
&\mc{1}{c|}{}
&\mc{1}{c|}{}
\\[-5pt]
%%%%%%%%%%%%%%%%%%%%%%%%%%%% 2nd line %%%%%%%%%%%%%%%%%%%%%%%%%%%%%%%%%%
\mc{2}{|c|}{}
&\mc{1}{c|}{\bf iteration}
&\mc{1}{c|}{$\eta_{\rm SDP}$}
&\mc{1}{c|}{$\eta{gap}$}
&\mc{1}{c|}{\bf CPU time}
\\[2pt]
\hline
%%%%%%%%%%%%%%%%%%%%%%%%%%%% 3rd line %%%%%%%%%%%%%%%%%%%%%%%%%%%%%%%%%%
\mc{1}{|c}{\bf problem}&\mc{1}{c|}{$m|n_E|n_I$}
&\mc{1}{c|}{\bf D-E$|$P1.0$|$P1.6$|$P1.9$|$iP1.0$|$iP1.6$|$iP1.9}
&\mc{1}{c|}{\bf D-E$|$P1.0$|$P1.6$|$P1.9$|$iP1.0$|$iP1.6$|$iP1.9}
&\mc{1}{c|}{\bf D-E$|$P1.0$|$P1.6$|$P1.9$|$iP1.0$|$iP1.6$|$iP1.9}
&\mc{1}{c|}{\bf D-E$|$P1.0$|$P1.6$|$P1.9$|$iP1.0$|$iP1.6$|$iP1.9}
\\
\hline
\endfirsthead
\multicolumn{3}{c}{Table~\thetable~(continued) }\vspace{0.5em}
\\
\hline
%%%%%%%%%%%%%%%%%%%%%%%%%%%%%%%%%%%%%%%%%%%%%%%%%%%%%%%%%%%%%%
\mc{2}{|c|}{}
&\mc{1}{c|}{}
&\mc{1}{c|}{}
&\mc{1}{c|}{}
&\mc{1}{c|}{}
\\[-5pt]
%%%%%%%%%%%%%%%%%%%%%%%%%%%% 2nd line %%%%%%%%%%%%%%%%%%%%%%%%%%%%%%%%%%
\mc{2}{|c|}{}
&\mc{1}{c|}{\bf iteration}
&\mc{1}{c|}{$\eta_{\rm SDP}$}
&\mc{1}{c|}{$\eta{gap}$}
&\mc{1}{c|}{\bf CPU time}
\\[2pt]
\hline
%%%%%%%%%%%%%%%%%%%%%%%%%%%% 3rd line %%%%%%%%%%%%%%%%%%%%%%%%%%%%%%%%%%
\mc{1}{|c}{\bf problem}&\mc{1}{c|}{$m|n_E|n_I$}
&\mc{1}{c|}{\bf D-E$|$P1.0$|$P1.6$|$P1.9$|$iP1.0$|$iP1.6$|$iP1.9}
&\mc{1}{c|}{\bf D-E$|$P1.0$|$P1.6$|$P1.9$|$iP1.0$|$iP1.6$|$iP1.9}
&\mc{1}{c|}{\bf D-E$|$P1.0$|$P1.6$|$P1.9$|$iP1.0$|$iP1.6$|$iP1.9}
&\mc{1}{c|}{\bf D-E$|$P1.0$|$P1.6$|$P1.9$|$iP1.0$|$iP1.6$|$iP1.9}
\\[1mm]
\hline
\endhead
\hline
\endfoot
\input{table-QSDP.txt}
\end{longtable}
\end{center}

\end{landscape}

\end{document}

%% file: table-SDP.txt
theta4 &1949$|$200&408$|$344$|$314$|$601$|$3401&9.8-7$|$9.4-7$|$9.9-7$|$8.9-7$|$8.6-7&9.7-8$|$-5.4-7$|${\green{1.7-6}}$|$-7.3-10$|$-5.5-9
&2.6$|$1.7$|$1.6$|$2.8$|$15.1
\\[2pt]  
theta42 &5986$|$200&205$|$169$|$173$|$166$|$167&9.8-7$|$9.5-7$|$9.5-7$|$9.9-7$|$9.6-7&-4.0-8$|$-4.5-8$|$4.3-8$|$3.9-8$|$3.5-8
&1.1$|$0.9$|$0.9$|$0.9$|$0.9
\\[2pt]  
theta6 &4375$|$300&400$|$353$|$333$|$672$|$3735&9.0-7$|$9.0-7$|$9.4-7$|$9.9-7$|$9.9-7&{\green{-1.7-6}}$|${\green{2.0-6}}$|${\green{2.1-6}}$|$-2.0-10$|$1.5-9
&4$|$3.6$|$3.3$|$6.7$|$35.9
\\[2pt]  
theta62 &13390$|$300&186$|$185$|$173$|$163$|$201&9.5-7$|$9.2-7$|$9.6-7$|$9.6-7$|$9.7-7&5.4-8$|$1.4-7$|$1.4-7$|$5.8-9$|$-6.4-10
&2$|$2$|$2.1$|$1.8$|$2.1
\\[2pt]  
theta8 &7905$|$400&453$|$380$|$357$|$644$|$3401&9.4-7$|$8.1-7$|$9.7-7$|$9.9-7$|$8.8-7&-6.5-7$|${\green{1.7-6}}$|${\green{2.0-6}}$|$-2.1-10$|$1.7-9
&9.4$|$7.9$|$7.4$|$13.3$|$1:10
\\[2pt]  
theta82 &23872$|$400&191$|$171$|$172$|$181$|$182&9.1-7$|$9.7-7$|$9.2-7$|$9.5-7$|$9.8-7&-1.5-9$|$1.8-7$|$6.8-8$|$7.1-8$|$7.2-8
&4.3$|$3.8$|$3.8$|$4$|$4
\\[2pt]  
theta83 &39862$|$400&173$|$161$|$150$|$147$|$147&9.9-7$|$9.5-7$|$9.8-7$|$9.7-7$|$9.4-7&-6.2-8$|$1.1-7$|$1.1-7$|$1.1-7$|$1.1-7
&3.9$|$3.7$|$3.4$|$3.4$|$3.3
\\[2pt]  
theta10 &12470$|$500&464$|$386$|$389$|$686$|$2801&8.5-7$|$8.9-7$|$8.2-7$|$9.9-7$|$9.4-7&{\green{-1.4-6}}$|${\green{2.0-6}}$|${\green{1.7-6}}$|$-1.2-10$|$2.9-9
&14.6$|$12.2$|$12.3$|$21.5$|$1:27
\\

%% file: table-QSDP.txt
be100.1 &101$;$101$;$14850 &26599$|$25983$|$20453$|$19877$|$6585$|$5190$|$4912
&9.9-7$|$9.9-7$|$9.9-7$|$9.9-7$|$9.9-7$|$9.9-7$|$9.9-7
&1.7-7$|$1.8-7$|$8.6-8$|$4.8-8$|${\green{-1.2-6}}$|$-8.7-7$|${\green{-1.2-6}}
&1:11$|$1:19$|$1:01$|$1:00$|$33.5$|$22.8$|$22
\\[2pt]  \hline 
be100.2 &101$;$101$;$14850 &14616$|$14399$|$11747$|$11562$|$5292$|$4442$|$4287
&9.9-7$|$9.9-7$|$9.9-7$|$9.9-7$|$9.9-7$|$9.9-7$|$9.9-7
&4.4-7$|$3.5-7$|$3.4-7$|$3.6-8$|$-9.7-7$|$-4.8-7$|$-7.0-7
&43.7$|$44.1$|$35.4$|$34.9$|$22$|$19.1$|$18.4
\\[2pt]  \hline 
be100.3 &101$;$101$;$14850 &17930$|$17453$|$13747$|$11606$|$4230$|$3520$|$3414
&9.9-7$|$9.9-7$|$9.9-7$|$9.9-7$|$9.9-7$|$9.9-7$|$9.9-7
&-1.6-7$|$-1.5-7$|$-1.5-7$|$-1.4-7$|$-4.4-7$|$-3.2-7$|$-2.4-7
&52.3$|$50.5$|$40.6$|$34.1$|$18.5$|$15.5$|$14.8
\\[2pt]  \hline 
be100.4 &101$;$101$;$14850 &28987$|$30096$|$23233$|$21379$|$6856$|$5214$|$4916
&9.9-7$|$9.9-7$|$9.9-7$|$9.9-7$|$9.9-7$|$9.9-7$|$9.9-7
&1.8-7$|$1.8-7$|$1.6-7$|$1.6-7$|$-3.6-7$|$-8.8-7$|$-5.0-7
&1:24$|$1:30$|$1:08$|$1:02$|$31.1$|$22.6$|$22.2
\\[2pt]  \hline 
be100.5 &101$;$101$;$14850 &13962$|$15264$|$10192$|$10004$|$4593$|$3686$|$3612
&9.9-7$|$9.9-7$|$9.9-7$|$9.9-7$|$9.9-7$|$9.9-7$|$9.9-7
&8.8-7$|${\green{1.0-6}}$|${\green{1.0-6}}$|$9.3-7$|$-4.2-7$|$7.1-8$|$-3.9-7
&41.7$|$46.3$|$30.9$|$30.8$|$19.2$|$16.2$|$15.9
\\[2pt]  \hline 
be100.6 &101$;$101$;$14850 &18215$|$18685$|$14235$|$13690$|$5209$|$4049$|$4287
&9.9-7$|$9.9-7$|$9.9-7$|$9.9-7$|$9.9-7$|$9.9-7$|$9.9-7
&6.5-7$|$6.5-7$|$6.4-7$|$5.7-7$|${\green{-1.1-6}}$|${\green{-1.2-6}}$|$-6.9-7
&53.1$|$54.3$|$42.1$|$40.2$|$22.4$|$18.1$|$19.2
\\[2pt]  \hline 
be100.7 &101$;$101$;$14850 &20092$|$20280$|$15472$|$14825$|$5281$|$4113$|$3952
&9.9-7$|$9.9-7$|$9.9-7$|$9.9-7$|$9.9-7$|$9.9-7$|$9.9-7
&2.8-7$|$3.1-7$|$2.1-7$|$2.4-7$|$-9.2-7$|$-5.1-7$|$-6.1-7
&59.6$|$59.5$|$45.9$|$43.4$|$23.1$|$18.3$|$17.7
\\